\documentclass[11pt,reqno]{amsart}
\usepackage{amsthm}
\usepackage{amssymb}
\usepackage{latexsym}
\usepackage{multicol}
\usepackage{verbatim,enumerate}
\usepackage{amscd}
\usepackage[latin1]{inputenc}
\usepackage{hyperref}
\usepackage{amsmath, amscd}
\usepackage[usenames]{color}

\advance\textwidth by 1.2in \advance\oddsidemargin by -.6in \advance\evensidemargin by -.6in
\newtheorem*{cor}{Corollary}
\newtheorem*{lem}{Lemma}
\newtheorem*{prop}{Proposition}

\theoremstyle{definition}
\newtheorem*{defn}{Definition}
\theoremstyle{definition}
\newtheorem{thm}{Theorem}

\newtheorem*{rem}{Remark}

\newenvironment{pf}{\proof}{\endproof}
\newcounter{cnt}
\newenvironment{enumerit}{\begin{list}{{\hfill\rm(\roman{cnt})\hfill}}{%
\settowidth{\labelwidth}{{\rm(iv)}}\leftmargin=\labelwidth%
\advance\leftmargin by \labelsep\rightmargin=0pt\usecounter{cnt}}}{\end{list}} \makeatletter
\def\mydggeometry{\makeatletter\dg@YGRID=1\dg@XGRID=20\unitlength=0.003pt\makeatother}
\makeatother \theoremstyle{remark}

\numberwithin{equation}{section}

\let\bwdg\bigwedge
\def\bigwedge{{\textstyle\bwdg}}

\begin{document}

\newcommand{\thmref}[1]{Theorem~\ref{#1}}
\newcommand{\secref}[1]{Section~\ref{#1}}
\newcommand{\lemref}[1]{Lemma~\ref{#1}}
\newcommand{\propref}[1]{Proposition~\ref{#1}}
\newcommand{\corref}[1]{Corollary~\ref{#1}}
\newcommand{\remref}[1]{Remark~\ref{#1}}
\newcommand{\defref}[1]{Definition~\ref{#1}}
\newcommand{\er}[1]{(\ref{#1})}
\newcommand{\id}{\operatorname{id}}
\newcommand{\ord}{\operatorname{\emph{ord}}}
\newcommand{\sgn}{\operatorname{sgn}}
\newcommand{\wt}{\operatorname{wt}}
\newcommand{\tensor}{\otimes}
\newcommand{\from}{\leftarrow}
\newcommand{\nc}{\newcommand}
\newcommand{\rnc}{\renewcommand}
\newcommand{\dist}{\operatorname{dist}}
\newcommand{\qbinom}[2]{\genfrac[]{0pt}0{#1}{#2}}
\nc{\cal}{\mathcal} \nc{\goth}{\mathfrak} \rnc{\bold}{\mathbf}
\renewcommand{\frak}{\mathfrak}
\newcommand{\supp}{\operatorname{supp}}
\newcommand{\Irr}{\operatorname{Irr}}
\renewcommand{\Bbb}{\mathbb}
\nc\bomega{{\mbox{\boldmath $\omega$}}} \nc\bpsi{{\mbox{\boldmath $\Psi$}}}
 \nc\balpha{{\mbox{\boldmath $\alpha$}}}
 \nc\bpi{{\mbox{\boldmath $\pi$}}}
 \nc\bvpi{{\mbox{\boldmath $\varpi$}}}
 \nc\bxi{{\mbox{\boldmath $\xi$}}}
  \nc\bmu{{\mbox{\boldmath $\mu$}}}
\nc\bsigma{{\mbox{\boldmath $\sigma$}}} \nc\bcN{{\mbox{\boldmath $\cal{N}$}}} \nc\bcm{{\mbox{\boldmath $\cal{M}$}}} \nc\bLambda{{\mbox{\boldmath
$\Lambda$}}}
\def\tqbinom#1#2{\text{$\left[\begin{smallmatrix} #1\\#2\end{smallmatrix}\right]$}}
\newcommand{\lie}[1]{\mathfrak{#1}}
\makeatletter
\def\section{\def\@secnumfont{\mdseries}\@startsection{section}{1}%
  \z@{.7\linespacing\@plus\linespacing}{.5\linespacing}%
  {\normalfont\scshape\centering}}
\def\subsection{\def\@secnumfont{\bfseries}\@startsection{subsection}{2}%
  {\parindent}{.5\linespacing\@plus.7\linespacing}{-.5em}%
  {\normalfont\bfseries}}
\makeatother
\def\subl#1{\subsection{}\label{#1}}
 \nc{\Hom}{\operatorname{Hom}}
  \nc{\mode}{\operatorname{mod}}
\nc{\End}{\operatorname{End}} \nc{\wh}[1]{\widehat{#1}} \nc{\Ext}{\operatorname{Ext}} \nc{\ch}{\text{ch}} \nc{\ev}{\operatorname{ev}}
\nc{\Ob}{\operatorname{Ob}} \nc{\soc}{\operatorname{soc}} \nc{\rad}{\operatorname{rad}} \nc{\head}{\operatorname{head}}
\def\Im{\operatorname{Im}}
\def\gr{\operatorname{gr}}
\def\mult{\operatorname{mult}}
\def\Max{\operatorname{Max}}
\def\ann{\operatorname{Ann}}
\def\sym{\operatorname{sym}}
\def\Res{\operatorname{\br^\lambda_A}}
\def\und{\underline}
\def\Lietg{$A_k(\lie{g})(\bsigma,r)$}

 \nc{\Cal}{\cal} \nc{\Xp}[1]{X^+(#1)} \nc{\Xm}[1]{X^-(#1)}
\nc{\on}{\operatorname} \nc{\Z}{{\bold Z}} \nc{\J}{{\cal J}} \nc{\C}{{\bold C}} \nc{\Q}{{\bold Q}}
\renewcommand{\P}{{\cal P}}
\nc{\N}{{\Bbb N}} \nc\boa{\bold a} \nc\bob{\bold b} \nc\boc{\bold c} \nc\bod{\bold d} \nc\boe{\bold e} \nc\bof{\bold f} \nc\bog{\bold g}
\nc\boh{\bold h} \nc\boi{\bold i} \nc\boj{\bold j} \nc\bok{\bold k} \nc\bol{\bold l} \nc\bom{\bold m} \nc\bon{\bold n} \nc\boo{\bold o}
\nc\bop{\bold p} \nc\boq{\bold q} \nc\bor{\bold r} \nc\bos{\bold s} \nc\boT{\bold t} \nc\boF{\bold F} \nc\bou{\bold u} \nc\bov{\bold v}
\nc\bow{\bold w} \nc\boz{\bold z} \nc\boy{\bold y} \nc\ba{\bold A} \nc\bb{\bold B} \nc\bc{\bold C} \nc\bd{\bold D} \nc\be{\bold E} \nc\bg{\bold
G} \nc\bh{\bold H} \nc\bi{\bold I} \nc\bj{\bold J} \nc\bk{\bold K} \nc\bl{\bold L} \nc\bm{\bold M} \nc\bn{\bold N} \nc\bo{\bold O} \nc\bp{\bold
P} \nc\bq{\bold Q} \nc\br{\bold R} \nc\bs{\bold S} \nc\bt{\bold T} \nc\bu{\bold U} \nc\bv{\bold V} \nc\bw{\bold W} \nc\bz{\bold Z} \nc\bx{\bold
x} \nc\KR{\bold{KR}} \nc\rk{\bold{rk}} \nc\het{\text{ht }}

\nc\toa{\tilde a} \nc\tob{\tilde b} \nc\toc{\tilde c} \nc\tod{\tilde d} \nc\toe{\tilde e} \nc\tof{\tilde f} \nc\tog{\tilde g} \nc\toh{\tilde h}
\nc\toi{\tilde i} \nc\toj{\tilde j} \nc\tok{\tilde k} \nc\tol{\tilde l} \nc\tom{\tilde m} \nc\ton{\tilde n} \nc\too{\tilde o} \nc\toq{\tilde q}
\nc\tor{\tilde r} \nc\tos{\tilde s} \nc\toT{\tilde t} \nc\tou{\tilde u} \nc\tov{\tilde v} \nc\tow{\tilde w} \nc\toz{\tilde z} \nc\woi{w_{\omega_i}}
\newcommand{\nn}{\nonumber}

\title{Prime Representations from a Homological Perspective}
\author{Vyjayanthi Chari, Adriano Moura, Charles Young}
\address{Department of Mathematics, University of California, Riverside, CA 92521, U.S.A.}\email{vyjayanthi.chari@ucr.edu}\thanks{V.C. was partially supported by DMS-0901253}
\address{Departamento de Matemática, Universidade Estadual de Campinas, Campinas - SP - Brazil, 13083-859.}\email{aamoura@ime.unicamp.br}\thanks{A.M. was partially supported by CNPq grant 306678/2008-0}
\address{Department of Mathematics University of York, York, YO10 5DD, U.K.}
\email{casy500@york.ac.uk}
\thanks{C.A.S.Y was supported by EPSRC grant number EP/H000054/1}\maketitle
\begin{abstract}
We begin the study of simple finite-dimensional prime representations of quantum affine algebras from a homological perspective. Namely, we explore the relation between self extensions of simple representations and the property of being prime. We show that every nontrivial simple module has a nontrivial self extension. Conversely, if a simple representation has a unique nontrivial self extension up to isomorphism, then its Drinfeld polynomial is a power of the Drinfeld polynomial of a prime representation. It turns out that, in the sl(2) case, a simple module is prime if and only if it has a unique nontrivial self extension up to isomorphism. It is tempting to conjecture that this is true in general and we present a large class of prime representations satisfying this homological property.\end{abstract}

\section*{Introduction}The study of finite-dimensional representations of quantum affine algebras has been an active field of research
for at least two decades. The abstract  classification of the simple representations was given in \cite{CPqa}, \cite{CPbanff}, and  much of the  subsequent work has focused on understanding the structure of these  representations. This has proved to be a difficult task and a complete understanding outside the case of $\lie{sl}_2$ is still some distance away. A number of important methods have been developed:  for instance, the work of \cite{FM}, \cite{FR} on $q$--characters  has resulted in a deeper combinatorial  understanding of these representations. The geometric approach of H. Nakajima and the theory of crystal bases  of M. Kashiwara have also been very fruitful.
Another powerful tool is the T--system \cite{Her06,KNS,Nak03},
 which was recently shown \cite{MY} to extend beyond Kirillov-Reshetikhin modules to wider classes of representations.
A connection with the theory of cluster algebras has been established recently in \cite{herlec,Nak:cluster}.

The study of the structure of the irreducible representations can be reduced to the so called prime ones, namely those simple representations which cannot be written as a tensor product of two non-trivial simple representations.
Clearly any finite--dimensional simple representation can be written as a tensor product of simple prime representations and one could then focus on understanding the prime representations. This was the approach used in \cite{CPqa} for the $\lie{sl}_2$-case, but  generalizing this approach is very difficult. However, many examples of prime representations are known in general, for example the Kirillov--Reshethikhin modules are prime and, more generally, the minimal affinizations are also prime and other examples may be found for instance in \cite{herlec}. However, except in the $\lie{sl}_2$-case where the simple prime representations are precisely the Kirillov-Reshetikhin modules (which are also the evaluation modules), the classification of the prime representations is not known.

 This paper is motivated by an effort to understand the simple prime representations via homological properties. Thus, let $\hat{\cal F}$ be the category of finite-dimensional representations of the quantum affine algebra and denote by $V(\bpi)$ the irreducible representation associated to the Drinfeld polynomial $\bpi$. We construct in a natural way a non-trivial self-extension of any object $V$ of $\hat{\cal F}$ which motivates the natural question of characterizing  the simple objects which satisfy \begin{equation}\label{intro1}\tag{1}\dim\Ext^1_{\hat{\cal F}}(V, V)=1.\end{equation}
Our first result shows that any simple $V$ satisfying \eqref{intro1} is of the form $V(\bpi_0^s)$ for some $s\ge 1$ where $\bpi_0$ is such that $V(\bpi_0)$ is prime. Hence,  if $V(\bpi_0)$ is a real prime in the sense of \cite{herlec}, then using \cite{her:simpletp} we see that $V$ is a tensor power of  isomorphic of  $V(\bpi_0)$.

In the case of $\lie{sl}_2$  we prove the stronger result that a simple object  $V$ satisfies \eqref{intro1} if and only if $V$ is prime. It is natural and now obviously interesting to ask if such a result remains true for general $\lie g$. Our next result provides partial evidence for this to be true. Namely, we prove for a large family of simple prime representations including the minimal affinizations   that the space of self-extensions is one-dimensional. Our results go beyond minimal affinizations and we  prove that the representations $S(\beta)$ defined in \cite{herlec} have a one--dimensional space of extensions  as long as $\beta$ is a positive root in which every simple root occurs with multiplicity one. 
It is worth comparing the results of this paper with their  non-quantum counter parts. One can define in a similar way the notion of prime representations for the category  of finite-dimensional representations of an affine Lie algebra $\hat{\lie g}$. It is known through the work of \cite{CG}, \cite{Kodera}  that if $V,V'$ are irreducible finite-dimensional representations  of $\hat{\lie g}$, then  $$\Ext^1_{\hat{\lie g}}(V,V')\cong\Hom_{\lie g}(\lie g\otimes V, V').$$
 It is now easily seen that there exist examples of simple prime representations $V$ such that $\Ext^1_{\hat{\lie g}}(V,V)$ has dimension at least two. In Section 2, we give an example of a simple representation of the quantum affine algebra which has a one--dimensional space of self--extensions but whose classical limit, although also prime and simple, has a two dimensional space of self--extensions.

The paper is organized as follows. In Section \ref{s:pre} we recall the definition and some standard results on quantum affine algebras. In Section \ref{s:main} we give  review some results  on finite-dimensional representations of the quantum affine algebra and state the main results of the paper. In Section \ref{s:E(V)} we construct a self extension of any given module of the quantum affine algebra.  We prove that if the  module is simple and finite-dimensional  then the extension is nontrivial. We also give the previously mentioned condition on the Drinfeld polynomials of a simple representation  which  satisfy \eqref{intro1}.

In Section \ref{s:weyl}   we first review  results on local and global Weyl modules. We then  study the relationship  between  these modules and self extensions of simple representations. In particular, we compute the dimension of the space of self extensions of the local Weyl modules. In Section \ref{s:tp} we prove that the space of self extensions of a simple module is a subspace of the self--extensions of the corresponding Weyl module. This allows us to prove that a self--extension of a simple module is ``determined'' by its top weight space and this plays a crucial role in proving the remaining results of the paper. It allows  to study the relationship between self extensions and tensor products of simple modules which in turn establishes the condition for a simple representation to have a one--dimensional space of self--extensions.   In Section \ref{s:sl2} we prove that, in the $\lie{sl}_2$ case, a simple module $V$ satisfies \eqref{intro1} if and only if $V$ is prime.
The last section is dedicated to the proof that a certain class of  simple modules satisfy \eqref{intro1}. Our results show that this implies that these modules are prime. The latter fact that the modules are prime can also be proved by other methods as well, see \cite{herlec} for the modules of type $S(\beta)$ and \cite{Cher} for remarks on  minimal affinizations. Our goal here is really to provide evidence towards the conjecture for general $\lie g$, that  $V$ satisfies  \eqref{intro1} iff $V$ is prime.
\section{Preliminaries}\label{s:pre}

Throughout the paper  $\bc$  (resp. $\bz$, $\bz_+$) denotes the set of complex  numbers (resp. integers, non--negative integers) and $\bc^\times$ (resp. $\bz^\times$ ) is the set of non--zero complex numbers   (resp. integers).

\subsection{}

Let $I=\{1,\cdots,n\}$ be the index set for the set of simple roots $\{\alpha_i: i\in I\}$ of  an irreducible reduced  root system $R$ in a real vector space. Let $R^+$ be the corresponding set of positive roots. Fix a set of fundamental weights $\{\omega_i:i\in I\}$ and  let $Q$,  $P$ be the associated root and weight lattice respectively and recall that $Q\subset P$.  If $R^+$ is the set of positive roots then we let $Q^+$ be the $\bz_+$--span of $R^+$ and $P^+$ the $\bz_+$--span of the fundamental weights.

We assume that the nodes of the Dynkin diagram are numbered as in \cite{Bourbaki} and we follow the conventions of that labeling.
Let $\hat I=I\cup\{0\}$ be the nodes of the corresponding (untwisted) extended Dynkin diagram and denote by  $A=(a_{i,j})_{i,j\in I}$ (resp. $\hat A=(a_{i,j})_{i,j\in \hat I}$)  the associated Cartan (resp. untwisted affine Cartan) matrix.  Finally, fix non--negative integers $\{d_i: i\in\hat I\}$ such that the matrix $(d_ia_{i,j})_{i,j\in\hat I}$ is symmetric.

From now on, we fix $q\in \bc^\times$ and assume that $q$ is not a root of unity. For   $m\in\bz$, $\ell,r\in\bz_+$ and $i\in\hat I$, set $q_i=q^{d_i}$ and define,
\begin{gather*}[m]_i =  \frac{q_i^m-q_i^{-m}}{q_i-q_i^{-1}}\qquad
[0]_i!=1, \qquad [\ell]_i!= [\ell]_i[\ell-1]_i\dots[1]_i,\\ \\
\qbinom{\ell}{r}_i= \frac{[\ell]_i!}{[\ell-r]![r]_i!}.\end{gather*}

\subsection{} Let  $\hat{\bu}_q$ (resp. $\bu_q$) be the associative algebra over $\bc$
with generators $x_i^{\pm}, k_i^{\pm 1}, i\in\hat I$, (resp. $x_i^{\pm}, k_i^{\pm 1}, i\in I$) satisfying the following defining relations: for  $i,j\in \hat I$ (resp. $i,j\in I$), we have
\begin{gather*}
k_ik_i^{-1} = 1, \quad k_ik_j=k_jk_i, \\
 k_ix_j^{\pm}k_i^{-1} = q_i^{\pm a_{i,j}}x_j^{\pm},\\  [x_i^+, x_j^-] = \delta_{ij}\frac{k_i -k_i^{-1}}{q_i - q_i^{-1}}\\\sum_{m=0}^{1-a_{ij}}(-1)^m\tqbinom{1- a_{ij}}{m}_{i}(x^{\pm}_i)^{1-a_{i,j}-m}x_j^{\pm}(x_i^{\pm})^m = 0, \ \ i\ne j.
\end{gather*}
It is well--known that $\hat\bu_q$ and $\bu_q$ are Hopf algebras with counit,  comultiplication,  and antipode given as follows: for $i,j\in\hat I$ (resp. $i,j\in I$),
\begin{gather*}
\epsilon(k_i)  = 1, \qquad \epsilon (x_i^{\pm}) = 0,\\
\Delta(k_i) = k_i\otimes k_i,\\
  \Delta(x_i^+) = x_i^+ \otimes 1 + k_i \otimes x_i^+,\qquad \Delta(x_i^-) = x_i^-\otimes k_i^{-1} + 1 \otimes x_i^-,\\
S(k_i) = k_i^{-1},  \qquad S(x_i^+) = - k_i^{-1}x_i^+, \quad S(x_i^-) = -x_i^-k_i.
\end{gather*}
It is known  that $\bu_q$ can be canonically identified with the subalgebra of $\hat\bu_q$ generated by the elements $x_i^\pm$, $k_i$, $i\in I$ and that  $\bu_q$ is a Hopf subalgebra of $\hat\bu_q$. The algebra $\hat\bu_q$ is naturally $\bz$--graded by requiring $$\gr x^\pm_i= 0, \quad i\in I, \quad\ \ \gr x_0^\pm=\pm 1.$$

\subsection{}The quantum loop algebra is the quotient of  $\hat{\bu}_q$ by the two sided ideal generated by $$k_0\prod_{i\in I}k_i^{r_i}- 1,$$ where $r_i\in\bz_+$ is the coefficient of $\alpha_i$ in the highest root of $R^+$. It is clearly both a  $\bz$--graded and  a  Hopf ideal and hence the quantum loop algebra also acquires  a $\bz$--grading and the structure of a Hopf algebra.

{\em From now on, we shall only be concerned with the quantum loop algebra and hence we shall by abuse of notation write $\hat\bu_q$ for the quantum loop algebra}.

The  algebra $\hat\bu_q$ has an alternative presentation given in \cite{Dr}, \cite{Beck}. It is the algebra with generators $x^{\pm}_{i,r}$, $h_{i,s}$, $k_i^{\pm1}$, where $ i\in I$, $r\in\bz$, $s\in\bz^{\times}$, and defining relations: for $i,j\in I$, $r,\ell\in\bz$, $s\in\bz^\times$,
\begin{gather*}
k_ik_i^{-1}=1, \quad k_ik_j=k_jk_i, \quad k_ih_{j,s}=h_{j,s}k_i, \quad h_{i,r}h_{j,s}=h_{j,s}h_{i,s} \quad ,\\
[h_{i,s}, x^{\pm}_{j,r}]=\pm \frac{1}{s}\ [sa_{i,j}]_{q_i} x^{\pm}_{j,r+s} , \qquad k_ix^{\pm}_{j,r}k^{-1}_i=q_i^{\pm a_{i,j}}x^{\pm}_{j,r}, \\
x^{\pm}_{i,r}x^{\pm}_{j,\ell}-q_i^{\pm a_{i,j}}x^{\pm}_{j,\ell}x^{\pm}_{i,r}= q_i^{\pm a_{i,j}}x^{\pm}_{i,r-1}x^{\pm}_{j,\ell+1}-x^{\pm}_{j,\ell+1}x^{\pm}_{i,r-1},\\
 [x^+_{i,r}, x^-_{j,s}]=\delta_{ij}\ \frac{\phi^+_{i,r+s}-\phi^-_{i,r+s}}{q_i-q_i^{-1}},
 \end{gather*} where $\phi_{i,\mp m}^\pm =0,\ \ m>0,$ and the elements $\phi^{\pm}_{i,\pm m}$, $m\ge 0$, are defined by the following equality of power series in $u$:
$$\sum_{m\geq 0}\phi^{\pm}_{i,\pm m}u^{m}=k^{\pm1}_i\exp\left(\pm (q_i-q_i^{-1}) \sum_{s>0}h_{i,\pm s}u^s \right).$$ Finally, for $i\ne j$ and given  $r_m\in\bz$ for $1\le m\le 1-a_{i,j}$, we have
\begin{gather*}
\sum_{\sigma \in S_{1-a_{i,j}}}\sum_{m=0}^{1-a_{i,j}}(-1)^{m}\tqbinom{1-a_{i,j}}{m}_{q_i}x^{\pm}_{i,r_{\sigma(1)}}\ldots x^{\pm}_{i,r_{\sigma(m)}}x^{\pm}_{j,s}
x^{\pm}_{i,r_{\sigma(m+1)}}\ldots x^{\pm}_{i,r_{\sigma(1-a_{i,j})}}=0,
 \end{gather*}
 where for any $k\in\bz_+$ we denote by $S_k$ the symmetric group on $k$ letters.
\vskip 12pt
The $\bz$--grading on $\hat\bu_q$ is the same as the one given by setting:$$\gr x_{i,r}^\pm =r,\ \ \gr h_{i,s}=s,\ \ \gr\phi^\pm_{i,\pm m}= \pm m,$$ where $i\in I$, $r,s\in\bz$, $s\ne 0$ and $m\in\bz_+$.

\subsection{} Let $\hat{\bu}_q^\pm$ be the subalgebra of $\hat{\bu}_q$ generated by the elements $\{x_{i,r}^\pm: i\in I, r\in\bz\}$, $\hat{\bu}_q(0)$ the subalgebra generated by $\{k_i^{\pm 1},h_{i,s}: i\in I, s\in\bz^\times\}$, and $\hat{\bu}_q^0$ the subalgebra of $\hat{\bu}_q(0)$  generated by $\{ h_{i,s}: i\in I, s\in\bz^\times\}$ .  Then we have an isomorphism of vector spaces,\begin{equation} \label{pbw} \hat{\bu}_q=\hat{\bu}_q^-\hat{\bu}_q(0)\hat{\bu}_q^+.\end{equation}
The algebra $\hat{\bu}_q(0)$ (resp. $\hat{\bu}_q^0$) is also generated by the elements $\phi_{i,m}^\pm$, $i\in I$, $m\in\bz$ (resp. $m\in\bz^\times$). We shall also need a third set of generators $\Lambda_{i,r}$, $ i\in I$, $r\in\bz$  for   $\hat{\bu}_q^0$. These were defined in \cite{CPqa} and are given by the following equality of power series,\begin{equation}\label{deflambda}
\sum_{r=0}^\infty \Lambda_{i,\pm r} u^{r}= \exp\left(-\sum_{s=1}^\infty\frac{h_{i,\pm s}}{[s]_{i}}u^{s}\right).
\end{equation}
In particular, $\Lambda_{i,0}=1$ for all $i\in I$ and $\gr\Lambda_{i,r}=r$. We conclude the section with the following result established in \cite{Beck}, \cite{BCP}.
\begin{lem}\label{poly} The algebra $\hat{\bu}_q^0$ is the polynomial algebra in the variables $\{h_{i,r}: i\in I, r\in\bz^\times\}$. Analogous statements hold for the generators $\{\phi^\pm_{i,m}: i\in I, m\in\bz^\times\}$ and $\{\Lambda_{i,r}: i\in I, r\in\bz^\times\}.$\hfill\qedsymbol
\end{lem}

\section{The Main Results}\label{s:main}

\subsection{} Let $\lie g$ be a simple Lie algebra with root system $R^+$. Let $\lie h$ be the Cartan subalgebra and $\bu(\lie g)$ the universal enveloping algebra of $\lie g$. In this case we can regard $Q$ and $P$ as lattices in the vector space dual of $\lie h$. Any finite--dimensional representation of $\lie g$ is semi-simple. i.e., is isomorphic to a direct sum of irreducible representations.  Further,  the irreducible finite--dimensional representations are indexed by $P^+$. For $\mu\in P^+$, let $\overline{V}(\mu)$ be an irreducible module associated to $\mu$. Then
$$\overline{V}(\mu)=\bigoplus_{\nu\in P}\overline{V}(\mu)_\nu,\ \ \overline{V}(\mu)_\nu=\{v\in\overline{V}(\mu): hv=\nu(h)v,\ \ h\in\lie h\}.$$

\subsection{} A representation $V$ of $\hat{\bu}_q$ is said to be of type 1 if \begin{equation}
V=\bigoplus_{\mu\in P}V_\mu,\ \ \qquad V_\mu = \{v\in V: k_iv=q_i^{\mu_i}v \text{ for all } i\in \hat I\}.
\end{equation} where we write $\mu=\sum_{i\in I}\mu_i\omega_i$. Set $\wt(V) = \{\mu\in P:V_\mu\ne 0\}$.
Let $\hat{\cal F}$ be the category of type 1 finite--dimensional representations of $\hat\bu_q$. Specifically, the objects of $\hat{\cal F}$  are type 1 representations of $\hat\bu_q$ and the  morphisms in the category are just $\hat\bu_q$--module maps. The category $\cal F$ of finite--dimensional representations of $\bu_q$ is defined similarly. Any $\hat\bu_q$--module can be regarded by restriction as a module for $\bu_q$ and we shall use this fact repeatedly without mention.  Since $\hat\bu_q$ and $\bu_q$ are Hopf algebras the categories $\hat{\cal F}$ and $\cal F$ contain the trivial one--dimensional representation and are closed under taking tensor products and  duals.
\begin{defn} We shall say that $V\in\Ob\hat{\cal F}$ is prime if either $V$ is trivial or if there does not exist nontrivial  $V_j\in\Ob\hat{\cal F}$, $j=1,2$, with  $V\cong V_1\otimes V_2$.
\hfill\qedsymbol
\end{defn}
Clearly any $V\in\Ob\hat{\cal F}$ can be written as a tensor product of prime representations. Following \cite{herlec} we shall say that a simple object $V$ in $\hat{\cal F}$ is a real prime if  $V$ is prime and $V^{\otimes 2}$ is irreducible. It is known through the work of \cite{her:simpletp} that for a real prime $V$ the object  $V^{\otimes r}$ is irreducible for all $r\ge 1$.
One can define prime objects in $\cal F$ in a similar way but this is not interesting as we shall now see.

\subsection{}   The following was proved in \cite{Lusztig}.
\begin{prop} The category $\cal F$ is semisimple. Given any $\mu\in P^+$, the $\bu_q$--module $V(\mu)$  generated by an element $v_\mu$  with relations:
$$x_i^+ v_\mu=0,\ \ k_iv_\mu=q_i^{\mu_i}v_\mu,\ \ \ (x_i^-)^{\mu_{i}+1}v_\mu=0,$$ is  a simple object of $\cal F$. Further, any simple object in $\cal F$ is isomorphic to $V(\mu)$ for some $\mu\in P^+$ and $$\dim V(\mu)_\nu=\dim\overline{V}(\mu)_\nu,\ \ \nu\in P.$$ Moreover, given $\lambda,\mu,\nu\in P^+$, we have $$\dim\Hom_{\lie g}(\overline{V}(\mu)\otimes \overline{V}(\lambda),\  \overline{V}(\nu))=\dim\Hom_{\cal F}({V(\mu)}\otimes {V(\lambda)}, {V(\nu)}).$$

\hfill\qedsymbol
\end{prop}
The following is now a consequence of the corresponding result for simple Lie algebras.
\begin{cor} \label{class} The representations $V(\mu)$, $\mu\in P^+$ are prime.
\hfill\qedsymbol
\end{cor}

\subsection{}  Given any type 1 module $V$  for $\hat{\bu}_q$, set $$V^+=\{v\in V: x_{i,r}^+v=0\ \ i\in I, r\in\bz\},\qquad \ V_\lambda^+=V^+\cap V_\lambda.$$ If $v\in V^+_\lambda$, then it follows from \eqref{pbw} that $\wt\hat{\bu}_qv\subset\lambda- Q^+.$ An element $v\in V$ is said to be an $\ell$--weight vector if it is a joint eigenvector for the $h_{i,r}$, $i\in I$, $r\in\bz^\times$.
An $\ell$--weight vector contained in $V_\lambda^+$ is called a highest-$\ell$-weight vector.  Notice that this is equivalent to requiring it to be a joint eigenvector for the $\phi^\pm_{i,m}$, $i\in I$, $m\in\bz$. A $\hat{\bu}_q$--module $V$ is said to be highest-$\ell$-weight  if it is generated by a highest-$\ell$-weight element. The following is well--known and  easily proved.
\begin{lem}\label{lweight} Let $V$ be a  highest-$\ell$-weight $\hat{\bu}_q$--module. Then $V$ has a unique irreducible quotient. Any simple object in $\hat{\cal F}$ is highest-$\ell$-weight and in fact the space  of highest-$\ell$-weight vectors is one--dimensional .\hfill\qedsymbol\end{lem}
If $V\in\Ob\hat{\cal F}$ is a highest-$\ell$-weight module generated by a highest-$\ell$-weight element $v$  then there are constraints imposed on the eigenvalues of $h_{i,r}$ on $v$.   We now explain these constraints
\subsection{} Let $u$ be an indeterminate, $\bc[u]$ the algebra of polynomials in $u$ with coefficients in $\bc$
 and  $\bc(u)$  the field of quotients. Let $\cal P^+$ be the the multiplicative monoid consisting of all $I$-tuples of the form $\bpi = (\pi_i)_{i\in I}$ where  $\pi_i$ is a polynomial in $\bc[u]$ with constant term $1$. The $I$--tuple consisting of the constant polynomial $1$ is called the trivial element of $\cal P^+$.

The following was proved in \cite{CPqa}, \cite{CPbanff}.
\begin{prop} Suppose that $V\in\Ob\hat{\cal F}$ is highest-$\ell$-weight with generator $v\in V_\lambda^+$ and assume that $\phi^\pm_{i,m} v=d_{i,m}v$, $i\in I$, $m\in\bz$. There exists an element $\bpi\in\cal P^+$ such that, \begin{equation}\label{lweight2}\sum_{m\geq 0}d^{+}_{i, m}u^{m} =q^{\pm\deg\pi_i}\frac{\pi_i(q_i^{-1}u)}{\pi_i(q_iu)}=\sum_{m\geq 0}d^{-}_{i, -m}u^{-m},\qquad \lambda_i=\deg\pi_i, \end{equation} in the sense that the left- and right-hand terms are the Laurent expansions of the
middle term about $0$ and $\infty$, respectively. Equivalently, one has an equality of power series,
\begin{equation}\label{hs}
\pi_i^\pm(u)= \exp\left(-\sum_{s=1}^\infty\frac{h_{i,\pm s}(\bpi)}{[s]_{q_i}}u^{s}\right) =\sum_{r\ge 0}\Lambda_{i, \pm r}(\bpi) u^r,
\end{equation}
where $\pi_i^+(u)=\pi_i(u)$ and $\pi_i^-(u)= u^{\deg\pi_i}\pi_i(u^{-1})/\left(u^{\deg\pi_i}\pi_i(u^{-1})\right)(0).$

Conversely, given $\bpi\in\cal P^+$ there exists a unique (up to isomorphism) irreducible highest-$\ell$-weight object $V(\bpi)\in\Ob\hat{\cal F}$ which is generated by a highest-$\ell$-weight vector  $v(\bpi)$ of $\ell$-weight given by \eqref{lweight2}.

\hfill\qedsymbol
\end{prop}
\noindent We remark that the trivial representation corresponds to taking the trivial $n$--tuple.
\vskip 12pt
{\em From now on, we shall use the convention that given $\bvpi\in\cal P^+$,  the eigenvalue of $\phi^\pm_{i,m}$ on an $\ell$--weight vector with $\ell$--weight $\bvpi$ is denoted   $\phi^\pm_{i,m}(\bvpi)$, and  $h_{i,r}(\bvpi)$  and $\Lambda_{i,r}(\bvpi)$ are  defined similarly.}

\subsection{} The category $\hat{\cal F}$ unlike $\cal F$  is not semisimple and we shall be interested in understanding extensions in the category. Our focus in this paper is the space of self extensions of the simple objects of $\hat{\cal F}$. The trivial object $\bc$ satisfies $\Ext^1_{\hat{\cal F}}(\bc,\bc)=0$.
Since any $V\in\Ob\hat{\cal F}$ has a Jordan--Holder series, it follows that if all the Jordan--Holder factors of  $V\in\Ob\hat{\cal F}$ are trivial, then $V$ is isomorphic to a direct sum of copies of the trivial representation.
 Our first main result is:
\begin{thm} Let $\bpi\in\cal P^+$ be non--trivial. We have,
\begin{enumerit}
\item[(i)] $\dim \Ext^1_{\hat{\cal F}}(V(\bpi),V(\bpi))\ge 1.$
\item[(ii)] Suppose that    $\dim \Ext^1_{\hat{\cal F}}(V(\bpi),V(\bpi))=1$. Then there exists $\bpi_0\in\cal P^+$  and $s\in\bz_+$ such that $V(\bpi_0)$ is prime and $\bpi=\bpi_0^s$. In addition, there exists a partition $s_1\ge\cdots\ge s_k>0$ of $s$ such that $V(\bpi_0^{s_r})$ is prime for all $1\le r\le k$ and
$$V(\bpi)\cong V(\bpi_0^{s_1})\otimes\cdots\otimes V(\bpi_0^{s_k}). $$
In particular if $V(\bpi_0)$ is a real prime, this is equivalent to saying that $V(\bpi)$ is a tensor power of $V(\bpi_0)$.
\item[(iii)] Let $\bpi_1,\bpi_2\in\cal P^+$ and assume that $V(\bpi_1)\otimes V(\bpi_2)$ is irreducible. We have $$\dim \Ext_{\hat{\cal F}}(V(\bpi_1)\otimes V(\bpi_2),\ \ V(\bpi_1)\otimes V(\bpi_2))\ge \dim \Ext_{\hat{\cal F}}(V(\bpi_1),\ V(\bpi_1)).$$
\item[(iv)] If $R$ is a root system of type $A_1$  then $\dim \Ext_{\hat{\cal F}}(V(\bpi),V(\bpi))=1$ iff $V(\bpi)$ is prime.\end{enumerit}
\end{thm}
Let  $\hat{\cal F}_1$ be  the full subcategory of $\hat{\cal F}$ consisting of objects $V$ satisfying the following: if $V(\bpi)$ is a nontrivial Jordan--Holder factor of $V$ then $\dim \Ext_{\hat{\cal F}}(V(\bpi),V(\bpi))=1$. The category  $\hat{\cal F}_1$  is not closed under taking tensor products. However, we do have the following:
\begin{cor} If  $V(\bpi)\in\Ob\hat{\cal F}_1$ then all the prime factors of $V(\bpi)$ are also in $\hat{\cal F_1}$.
\end{cor}

\begin{rem} There is an important family of objects of $\hat{\cal F}$ known as local Weyl modules or standard modules which are known to be generically irreducible. The  reader is  referred to Section \ref{s:weyl} of this paper where we determine the dimension of the space of  self extensions of  these   modules.  In particular,
if $W$ is a local Weyl module then $\dim\Ext_{\hat{\cal F}}(W,W)=1$ if and only if $W$ is prime.
\end{rem}

\subsection{} It is  tempting to conjecture that part (iv) of Theorem 1 is true in general. The proof of (iv) uses very special properties of the quantum affine algebra associated to $A_1$ and these properties are known to be false in general. However, there are various well known families of representations of quantum affine algebras which have many nice properties (such as the Kirillov--Reshethikhin modules and minimal affinizations) which are either  known  to be  or easily proved to be prime. The next main result of this paper shows that for many of these families it is true that the space of self extensions is one--dimensional.

\subsection{} For $\bpi\in\cal P^+$ set $$\supp\bpi=\{i\in I:\pi_i\ne 1\}. $$ Given $i,j\in I$, let $[i,j]$ be the minimal connected subset of $I$ containing $i$ and $j$ and let $(i,j)=[i,j]\setminus\{i,j\}$.
We shall prove,
\begin{thm} \label{thm2}
 Let $\bpi\in\cal P^+$ and  assume  that,
 \begin{enumerit}
  \item[(i)] for $i\in\supp\bpi$ there exists $a_i\in\bc^\times$ such that $$\pi_i=(1-a_iuq_i^{m_i-1})(1-a_iuq_i^{m_i-3})\cdots(1-a_iuq_i^{-m_i+1}),\ \  \deg\pi_i=m_i,$$
\item[(ii)] for $i,j\in\supp\bpi$ with  $(i,j)\cap\supp\bpi=\emptyset$, we have either
 $$a_{i}=a_{j} q^{r_{i,j}},$$
 or, there exists $k\in\supp\bpi$ with $(j,k)\cap\supp\bpi=\emptyset= (i,k)\cap\supp\bpi$ and
  $$a_{j}=a_{k} q^{ r_{j,k}},\ \ a_i=a_kq^{ r_{i,k}},$$ where $  r_{i,j}=\pm\left(d_{i}m_{i} - \left(\frac12\sum_{\substack{k\ne \ell \in [i,j] }}d_ka_{k,\ell}\right)\  + d_j m_{j}\right).$
 \end{enumerit}
Then $$\dim\Ext_{\hat{\cal F}}(V(\bpi), V(\bpi))=1.$$\end{thm}
\begin{rem} Observe that the second assumption in (ii) is only possible when $R$ is of type $D$ or $E$.  Notice that it follows from Theorem 1 that the simple representations satisfying the conditions of Theorem 2 are prime.
\end{rem}

\subsection{} We now make several remarks to explain the assumptions on $\bpi$. The notion of minimal affinizations was introduced in \cite{Cmin} and studied further in \cite{CPmin1},\cite{CPmin2},\cite{CPmin3}, \cite{Md4}. The first condition on $\bpi$ requires that each component define a minimal affinization for the  quantum affine algebra associated to  $A_1$. The second condition requires that if we restrict our attention to the connected subset $I_s$  of $I$ whose intersection with $\supp\bpi$ is $\{i_s,i_{s+1}\}$, then the $I_s$--tuple $(\pi_{i_s},1,\cdots, 1, \pi_{i_{s+1}})$ defines  a minimal affinization for the quantum affine algebra associated to $I_s$. For ease of exposition, we are being a bit careless here in the case of the root system of type $D_n$ but this is taken care of later in the paper.

\subsection{} As an example, if $R^+$ is of type $A_3$, then   $\bpi =(1-u, 1-q^3u, 1-u)$ and $\bpi=(1-u,1-q^3u, 1-q^6u)$ both satisfy the conditions of Theorem 2. The latter polynomial defines a minimal affinization while the former does not.

\subsection{} In the special case when $\supp\bpi=\{i_1\}$, the associated module is called a Kirillov--Reshetikhin module and our result shows that the space of self extensions of this module is one--dimensional.
\subsection{} We conclude this section by comparing the statement of Theorem \ref{thm2} with known results
for the loop algebra of $L(\lie g)$. Let $t$ be an indeterminate. Then, $$L(\lie g)=\lie g\otimes \bc[t,t^{-1}],\ \ [x\otimes t^n,y\otimes t^m]=[x,y]\otimes t^{n+m}.$$ Let $\overline{\cal F}$ be the category of finite--dimensional representations of $L(\lie g)$. Then, the irreducible representations are again given by $n$--tuples of  polynomials $\overline
{\bpi}$. The structure of the simple representations is easily described in this case and, as a consequence, one also understand the irreducible prime objects in this category. As an example, if we take $\lie g$ to be $\lie{sl}_3$, then one knows that $$\overline{V}(1-u,1-u)\cong_{\lie g} \overline{V}(\omega_1+\omega_2),$$ and hence $\overline{V}(1-u,1-u)$ is prime. It follows from the work of \cite{CG} and \cite{Kodera}  that $$\dim\Ext_{\overline{\cal F}}(\overline{V}(1-u,1-u),\overline{V}(1-u,1-u))=2.$$

It can be shown that the classical limit (as $q\to 1$) of the representation $V(\bpi)$, where $\bpi=(1-u,1-q^3u)$, is the representation $\overline{V}(1-u,1-u)$. Theorem \ref{thm2} shows however that $$\dim\Ext_{\hat{\cal F}}(V(\bpi), V(\bpi))=1\ne \dim\Ext_{\overline{\cal F}}(\overline{V}(1-u,1-u),\overline{V}(1-u,1-u)).$$

\section{Proof   of Theorem 1:  parts $\rm{(i)}$ and  $\rm{(ii)}$}\label{s:E(V)}

\subsection{}    Let $\wt:\cal P^+\to P^+$ be defined by $$\wt\bpi=\sum_{i\in I}(\deg\pi_i)\omega_i.$$
We shall say that $V\in\Ob\hat{\cal F}$ is a self extension of $V(\bpi)$, $\bpi\in\cal P^+$ if $V$ has a Jordan--Holder series of length two with both constituents being isomorphic to $V(\bpi)$, or equivalently, if we have a short exact sequence $$0\to V(\bpi)\stackrel{\iota}\longrightarrow V \stackrel{\tau}\longrightarrow V(\bpi)\to 0,$$ of objects of $\hat{\cal F}$.
We say that $V$ is a trivial self extension if $V\cong V(\bpi)\oplus V(\bpi)$ and non--trivial otherwise. Notice that $$\dim V_{\wt\bpi}=2.$$ and that $\iota(v(\bpi))$ is a highest-$\ell$-weight vector in $V_{\wt\bpi}$. Suppose that  $V$ contains another linearly independent vector $v'$ which is also highest $\ell$-weight. Then $$\dim(\hat{\bu}_qv')_{\wt\bpi}=1,\ \ {\rm {and\ so}}  \ \ v\notin\hat{\bu}_qv'.$$  Hence $\iota(V(\bpi))\cap\hat{\bu}_qv'=\{0\}$ which  implies that $\hat{\bu}_qv'\cong V(\bpi)$ and that $V$ is a trivial self extension.  Summarizing, we have proved,
 \begin{lem}\label{nontriv} Let $V\in\Ob\hat{\cal F}$ be a self extension of $V(\bpi)$. Then $V$ is nontrivial if and only if $V_{\wt\bpi}$ has a unique (up to scalars) highest-$\ell$-weight vector. Moreover, if $V$ is nontrivial and $v\in V_{\wt\bpi}$ is not an $\ell$--weight vector, then $V=\hat{\bu}_qv$.
 \hfill\qedsymbol
 \end{lem}

\subsection{} A self--extension $V$ of $V(\bpi)$ defines an  element $[V]$ of  $\Ext^1_{\hat{\cal F}}(V(\bpi), V(\bpi))$. Moreover by Lemma \ref{nontriv}, we see that  $[V]=0$ iff $V$ is  the trivial self extension.
\begin{lem}\label{lindep} Let $V,V'$ be self extensions of $V(\bpi)$ for some $\bpi\in\cal P^+$. Then $V$ and $V'$ are isomorphic as $\hat{\bu}_q$--modules iff there exists $c\in\bc$ such that $[V]=c[V']$ as elements of $\Ext^1_{\hat{\cal F}}(V(\bpi), V(\bpi))$.
\end{lem}
\begin{pf}
 It is clear that if  there exists $c\in\bc$ with $[V]=c[V']$ as elements of $\Ext_{\hat{\cal F}}^1(V(\bpi), V(\bpi))$ then they must be isomorphic as $\hat{\bu}_q$--modules. For the converse let $\eta: V\to V'$ be an isomorphism of $\hat{\bu}_q$--modules.  Lemma \ref{nontriv} shows that  $V$ is a trivial self extension of $V(\bpi)$ iff $V'$ is also the trivial self--extension
  and in that case $[V]=[V']=0$.  So assume that they are both non--split i.e., the short exact sequences  $$0\to V(\bpi)\stackrel{\iota}\longrightarrow V \stackrel{\tau}\longrightarrow V(\bpi)\to 0, \qquad 0\to V(\bpi)\stackrel{\iota'}\longrightarrow V' \stackrel{\tau'}\longrightarrow V(\bpi)\to 0$$ are non--split. Then $\eta(\iota(v(\bpi))=a\iota'(v(\bpi)),$ for some $a\in\bc^\times$. Since $V(\bpi)$ is irreducible, it follows in fact that $$\eta(\iota(v)=a\iota'(v),\ \ {\rm{for \  all}}\ \ v\in V(\bpi).$$ This means that the short exact sequence
 \begin{equation}
\label{int} 0\to V(\bpi)\stackrel{\eta\circ\iota}\longrightarrow V'\stackrel{\tau'}\longrightarrow V(\bpi)\to 0,\end{equation}  defines the same equivalence class as $a^{-1}[V']$. Next, choose $v\in V$ such that $\tau(v)=v(\bpi)$. Since $V=\hat{\bu}_qv$ and  $\tau'(\eta(v))=bv(\bpi)$ for some $b\in\bc ^\times$ we get  $\tau'\circ \eta=b\tau$, i.e.,the sequence defines the same equivalence class as  $b^{-1}[V]$. Therefore, $a[V]=b[V']$ as required.
 \end{pf}

\subsection{} Recall that $\hat{\bu}_q$ is a $\bz$--graded algebra and for $r\in\bz$ let  $\hat{\bu}_q[r]$ be the $r$--th graded piece. Given any $V\in\Ob\hat{\cal F}$ let  $\be(V)\in\Ob\hat{\cal F}$ be defined by requiring   $$\be(V)=V\oplus V,$$ as vector spaces and the action of $\hat{\bu}_q$  given by extending linearly the assignment,$$ g_r(v,w)= (g_rv, rg_rv+g_rw),\qquad g_r\in \hat{\bu}_q[r],\ \  v,w\in V. $$ Clearly we have a short exact sequence of $\hat{\bu}_q$-modules \begin{equation}\label{can} 0\to V\stackrel{\iota}{\longrightarrow} \be(V) \stackrel{\tau}{\longrightarrow}V\to 0. \end{equation}
The following   proves part (i) of Theorem 1.
\begin{prop}\label{thm1i} If $\bpi\in\cal P^+$ is nontrivial,  $\be(V(\bpi))$ is a nontrivial self extension of $V(\bpi)$.
\end{prop}
\begin{pf} By Lemma \ref{nontriv} it suffices to prove that $\be(V(\bpi))_{\wt\bpi}$ has a one--dimensional space of highest-$\ell$-weight vectors. Clearly $\iota(v(\bpi))$ is a highest-$\ell$-weight vector, and
for $c_1,c_2\in\bc$, we have,\begin{equation}\label{action} \phi^\pm_{i,m}(c_1v(\bpi),\ c_2v(\bpi))=(\phi_{i,m}^\pm(\bpi) c_1v(\bpi), \ \ m\phi_{i,m}^\pm(\bpi) c_1v(\bpi)+\ \phi_{i,m}^\pm(\bpi) c_2v(\bpi)).\end{equation}
Choosing $i\in I$ and $m\in\bz$, $m>0$ with $\phi_{i,m}(\bpi)\ne 0$ we see that the right hand side of the preceding equation is a multiple of $(c_1v_\bpi, c_2v_\bpi)$ iff $c_1=0$ and the proposition is proved.
\end{pf}

\subsection{} To prove (ii) of Theorem 1, we need the following result which can be found in \cite{CPqa}. In the current formulation, it uses the formulae for the comultiplication given in \cite{Da},\cite{BCP}.
\begin{prop} Let $V_1, V_2\in\Ob\hat{\cal F}$. Let $v_1$ and $v_2$ satisfy $$x_{i,r}^+v_1=0= x_{i,r}^+v_2,\ \ i\in I, \ \ r\in\bz.$$ Then,\begin{gather*}\Delta(x_{i,r}^+)(v_1\otimes v_2)=0,\\ \\
\Delta(h_{i,s})(v_1\otimes v_2)=( h_{i,s}v_1\otimes v_2 +v_1\otimes h_{i,s}v_2),\end{gather*} for all $i\in I$, $r\in\bz$ and $s\in\bz^\times$.
\end{prop}
\begin{cor}\label{prod} Let $\bpi_j\in\cal P^+$ for $j=1,2$ and let $v=v(\bpi_1)\otimes v(\bpi_2)\in V(\bpi_1)\otimes V(\bpi_2)$. Then, $\hat{\bu}_qv$ is a highest-$\ell$-weight module with highest $\ell$-weight $\bpi_1\bpi_2$. In particular, it has $V(\bpi_1\bpi_2)$ as its unique irreducible quotient.
\end{cor}

\subsection{} \begin{lem}\label{nontriv2} Let $\bpi_1,\bpi_2\in\cal P^+$ be such that $V(\bpi_1)\otimes V(\bpi_2)$ is irreducible. If $V$ is a (non--trivial) self extension of $V(\bpi_1)$, then $V\otimes V(\bpi_2)$ is a (non--trivial) self extension of $V(\bpi_1)\otimes V(\bpi_2)$.  \end{lem}
\begin{pf} It is clear that $V\otimes V(\bpi_2)$ is a self extension of $V(\bpi_1)\otimes V(\bpi_2)$.  Let $v_1,v_2$ be a basis for $V_{\wt\bpi_1}$ and assume that $v_2$ is an $\ell$--weight vector. By Proposition \ref{prod} we have $$h_{i,r}(v_1\otimes v(\bpi_2) )= h_{i,r}v_1\otimes v(\bpi_2) +v_1\otimes h_{i,r}(\bpi_2)v(\bpi_2).$$ Hence $v_1\otimes v(\bpi_2)$ is an $\ell$--weight vector only if $h_{i,r}v_1\otimes v(\bpi_2)$ is a scalar multiple of $v_1\otimes v(\bpi_2)$, which implies that $h_{i,r}v_1$ is a scalar multiple of $v_1$. This  implies that $V_{\wt\bpi}$ has two linearly independent $\ell$--weight vectors and hence $V_1$ is  trivial by Lemma \ref{nontriv}.
\end{pf}

\subsection{} \begin{prop}\label{bevtensor}Let $\bpi_1,\bpi_2$ be nontrivial elements of $\cal P^+$ and assume that $(\bpi_1)^{r_1}\ne (\bpi_2)^{r_2}$ for all $r_1,r_2\in\bz_+$. Then $\be(V(\bpi_1))\otimes V(\bpi_2)$ and $V(\bpi_1)\otimes \be(V(\bpi_2))$ are non--trivial and non--isomorphic  self extensions of $V(\bpi_1)\otimes V(\bpi_2)$.
\end{prop}
\begin{pf}  The fact that $V(\bpi_1)\otimes \be(V(\bpi_2))$ and  $V(\bpi_1)\otimes \be(V(\bpi_2)))$ are non--trivial self extensions of $V(\bpi_1)\otimes V(\bpi_2)$ was proved in Lemma \ref{nontriv2}.
Suppose that  $\eta: \be(V(\bpi_1))\otimes V(\bpi_2)\to V(\bpi_1)\otimes \be(V(\bpi_2))$ is an isomorphism of $\hat{\bu}_q$--modules.  Then $\eta$ maps an $\ell$--weight vector in $(\be(V(\bpi_1))\otimes V(\bpi_2))_{\wt\bpi_1+\wt\bpi_2}$ to an $\ell$--weight vector in $(V(\bpi_1)\otimes \be(V(\bpi_2)))_{\wt\bpi_1+\wt\bpi_2}$ and hence we have,
$$\eta\left(\left(0,\ v(\bpi_1)\right)\otimes v(\bpi_2)\right)
= v(\bpi_1)\otimes \left(0,\ \ d v(\bpi_2)\right),$$ for some $d\in\bc$.  Moreover since $\eta$ is an isomorphism, we may and do assume without loss of generality that $d=1$.
Further, there also exist $c_1,c_2\in\bc$ such that
$$\eta\left(\left(v(\bpi_1),\ 0\right)\otimes v(\bpi_2)\right)
= v(\bpi_1)\otimes \left(c_1v(\bpi_2),\ c_2v(\bpi_2)\right).$$
By Proposition \ref{prod} again we  see that for all $i\in I$ and $r\in\bz$, $r\ne 0$, we have \begin{gather*} \eta\left( h_{i,r}\left( \left(v(\bpi_1),0\right)\otimes v(\bpi_2)\right)\right)
=\eta\left( h_{i,r}(\bpi_1) \left( v(\bpi_1), rv(\bpi_1)\right)\otimes v(\bpi_2)
+ h_{i,r}(\bpi_2)\left( v(\bpi_1),0\right)\otimes v(\bpi_2)\right)\\ = \left(h_{i,r}(\bpi_1)+h_{i,r}(\bpi_2)\right) v(\bpi_1)\otimes \left(c_1v(\bpi_2), \ c_2v(\bpi_2)\right)
+ rh_{i,r}(\bpi_1) v(\bpi_1)\otimes \left(0, v(\bpi_2)\right)\end{gather*} while,\begin{gather*}
h_{i,r}\left(v(\bpi_1)\otimes (c_1v(\bpi_2),\ c_2v(\bpi_2))\right) = h_{i,r}(\bpi_1) v(\bpi_1)\otimes  (c_1v(\bpi_2), c_2v(\bpi_2))+ \\  h_{i,r}(\bpi_2) v(\bpi_1)\otimes (c_1v(\bpi_2), \ c_1rv(\bpi_2)+c_2v(\bpi_2)).\end{gather*}
Equating, we get $$ h_{i,r}(\bpi_1)=c_1h_{i,r}(\bpi_2).$$ Writing $\bpi_1=(\pi_1,\cdots,\pi_n)$ and $\bpi_2=(\pi_1',\cdots,\pi_n')$, as $$\pi_i(u)=\prod_{s=1}^k(1-a_su)^{p_s},\ \ \pi_i'(u)=\prod_{s=1}^\ell(1-b_su)^{m_s},$$ where $a_s\ne a_r$ similarly $b_r\ne b_s$ if $r\ne s$ and $p_s>0$, $m_s>0$, we find by using \eqref{hs} that for all $r>0$, we have
 $$\frac{h_{i,r}(\bpi_1)}{[r]_i}=-\frac 1r\sum_{s=1}^k p_sa_s^r,\ \qquad \frac{h_{i,r}(\bpi_2)}{[r]_i}=-\frac 1r\sum_{s=1}^\ell m_sb_s^r.$$ Hence we get  $$\sum_{s=1}^k p_sa_s^r=c_1\sum_{s=1}^\ell m_sb_s^r,\ \ r\in\bz_+^\times.$$ If $a_1\ne b_r$ for all $1\le r\le \ell$, then we find by using the invertibility of the Vandermonde matrix that $p_1=0$ which is a contradiction. This means that we must have $\ell=k$ and also without loss of generality  $a_j=b_j$ for all $1\le j\le \ell$. This gives the equation, $$\sum_{s=1}^\ell(p_s-c_1m_s)a_s^r=0,\ \ r\in\bz, r\ne 0.$$ In particular, this means that $$p_s=c_1m_s,\ \ 1\le s\le \ell. $$ Hence $c_1$ is positive and  rational, say $c_1= \frac dd'$, for some $d,d'\in \bz_+\setminus \{0\},$and so $d'p_s= dm_s$ for all $1\le s\le\ell$. Since $c_1$ is independent of $i$ we have now proved that $$\bpi_1^{d}=\bpi_2^{d'}.$$ But this is a contradiction and hence $\eta$ is not an isomorphism.
\end{pf}
\subsection{}  We can now prove part (ii) of  Theorem 1. Observe first that if $\bpi_1$ and $\bpi_2$ are such that  $\bpi_1^{r_1}=\bpi_2^{r_2}$ for some $r_1, r_2\in\bz_+$, then there exists $\bpi_0$ such that $\bpi_1$ and $\bpi_2$ are powers of $\bpi_0$.
 Suppose that $V(\bpi)$ is not a   prime  representation and write $$V(\bpi)= V(\bpi_1)\otimes\cdots\otimes V(\bpi_k),$$ where $\bpi_s$ is nontrivial and $V(\bpi_s)$ is prime for $1\le s\le k$. Then by Proposition \ref{bevtensor} there exists $\bpi_0$ such that $$\bpi_1=\bpi_0^r,\ \ \bpi_2\cdots\bpi_k=\bpi_0^\ell,$$ (note $ V(\bpi_2)\otimes \cdots \otimes V(\bpi_k) \cong V(\bpi_2\cdots\bpi_k)$) and hence $\bpi=\bpi_0^{r+\ell}$. Since $$V(\bpi_1)\otimes V(\bpi_2)\cong V(\bpi_2)\otimes V(\bpi_1),$$ we also get by Proposition \ref{bevtensor} that $$\bpi_2^m=\bpi_1^p\bpi_3^p\cdots\bpi_k^p,$$ for some $m,p\in\bz_+$. This gives $$\bpi_2^{m+p}=\bpi^p=\bpi_0^{p(\ell+r)}.$$ Repeating we find that $\bpi_s$ is a power of $\bpi_0$ for all $1\le s\le k$ as required.

\section{Local  and global Weyl modules}\label{s:weyl} To prove the remaining results of the paper we need
 to recall the definition of global and local Weyl modules and summarize their important properties. We use the approach developed in \cite{CFK} for loop algebras.

\subsection{} Given $\lambda\in P^+$ the global Weyl module  $W(\lambda)$ is  the  $\hat{\bu}_q$--module generated by a vector $w_\lambda$ with  the following defining relations,
\begin{gather*}
k_iw_\lambda=q_i^{\lambda_i}w_\lambda, \qquad x_{i,r}^+w_\lambda=0,\\
(x_{i,0}^-)^{\lambda_i+1}w_\lambda=0,\end{gather*}
 for all $i\in I$ and $r\in\bz$. If $\lambda\ne 0$, then $W(\lambda)$ is an infinite--dimensional type 1 module while $W(0)$ is the trivial module. Note that if $\bpi\in\cal P^+$ is such that $\wt\bpi=\lambda$ and $V\in\Ob\hat{\cal F}$ is a highest-$\ell$-weight module with $\ell$-weight $\bpi$, then $V$ is a quotient of $W(\lambda)$.
The global Weyl module $W(\lambda)$ was originally defined in \cite{CPWeyl} in a different way, but it is not hard to see by using Proposition 4.3  of that paper  that the two definitions are equivalent.  It is also proved in Proposition 4.5 of \cite{CPWeyl} that $W(\lambda)$ is an integrable module: i.e the Chevalley generators $x_i^\pm$,  $i\in \hat I$ act locally nilpotently. Finally, we remark that it is proved in \cite{Na} that the global Weyl module is isomorphic to the extremal weight modules defined by Kashiwara in \cite{K}.
The following is a very special case of the  fact that $W(\lambda)$ is integrable. For all $i\in I$, we have \begin{equation}\label{intcons}\lambda-(\lambda_i+1)\alpha_i\notin \wt W(\lambda).\end{equation}

\subsection{}
Regard $W(\lambda)$  as a right-module for $\hat{\bu}_q^0$ by setting
$ (uw_\lambda)h_{i,r} =uh_{i,r}w_\lambda,$where $u\in\hat{\bu}_q$ and $i\in I$, $r\in\bz$, $r\ne 0$.

Set
\begin{equation*}
{\rm Ann}_\lambda = \{x\in \hat{\bu}_q^0: w_\lambda x = 0\} = \{x\in \hat{\bu}_q^0: xw_\lambda = 0\},
\end{equation*}
and let $\ba_\lambda$ be the quotient of $\hat{\bu}_q^0$ by the ideal ${\rm Ann}_\lambda$. Then, $W(\lambda)$ is a $(\hat{\bu}_q, \ba_\lambda)$-bimodule. For all $\mu\in P$ the subspace $W(\lambda)_\mu$ is a right $\ba_\lambda$--module. Moreover,   $W(\lambda)_\lambda$ is  obviously also a left  $\ba_\lambda$--module  and we have an isomorphism of  $\ba_\lambda$--bimodules $$W(\lambda)_\lambda\cong\ba_\lambda.$$
The structure of the ring $\ba_\lambda$ is known. Specifically regard $\hat{\bu}_q^0$ as the polynomial ring in the variables $\Lambda_{i,r}$, $i\in I$, $r\in\bz^\times$. Then $\ba_\lambda$ is the quotient obtained by setting \begin{equation}\label{defrelalambda}\Lambda_{i,r}=0,\ \ |r|\ge \lambda_i+1,
\ \ \Lambda_{i,\lambda_i}\Lambda_{i,-s}-\Lambda_{i,\lambda_i-s}=0,\end{equation} for all $i\in I$ and  $0\le s\le \lambda_i$. In particular, if we let $\bar\Lambda_{i,r}$ be the image of $\Lambda_{i,r}$ in $\ba_\lambda$, then $$\ba_\lambda\cong\bigotimes_{i\in I}\bc[\bar\Lambda_{i,1},\cdots\bar\Lambda_{i,\lambda_i},(\bar\Lambda_{i,\lambda_i})^{-1}].$$

The following important result will be crucial for the paper. The result was established when $R$ is $A_1$ in \cite{CPqweyl}, and in general the result can be deduced from the work of \cite{BN},\cite{N}. There are also other proofs of this result through the work of \cite{CPWeyl},\cite{CL},\cite{FoL},\cite{Naoi}.
\begin{thm}\label{free} The global Weyl module $W(\lambda)$ is a free right module of finite rank for $\ba_\lambda$.\hfill\qedsymbol
\end{thm}

\subsection{}\label{quotvlambda} Let $\mod$-$\ba_\lambda$ be the category of finitely generated left  $\ba_\lambda$--modules and given an object  $M$ of $\mod$-$\ba_\lambda$ set $$\bw_\lambda M= W(\lambda)\otimes_{\ba_\lambda} M.$$ Since $W(\lambda)_\lambda$ is an $\ba_\lambda$--bimodule, we have  an isomorphism of left $\ba_\lambda$--modules, $$(\bw_\lambda M)_\lambda= W(\lambda)_\lambda\otimes_{\ba_\lambda}M\cong M. $$
If  $V$ is any quotient of $W(\lambda)$, then $$u V_\lambda=0,\ \ u\in{\rm Ann}_\lambda,$$ and hence the $\hat{\bu}_q^0$ action on  $V_\lambda$ descends to $\ba_\lambda$. It is simple to check
 (see \cite[Proposition 3.6]{CFK}) that $V$ is a quotient of $\bw_\lambda V_\lambda$.

\subsection{} \begin{prop}\label{exact} The assignment $$M\to\bw_\lambda M,\ \ f\to 1\otimes f, $$ defines an exact functor from the category $\mod$-$\ba_\lambda$ to $\hat{\cal F}$. Moreover, $M$ is an  indecomposable object of $\mod$-$\ba_\lambda$ iff $\bw_\lambda M$ is indecomposable in $\hat{\cal F}$.
\end{prop}
\begin{pf} The first statement is clear from Theorem \ref{free}. Suppose that $$\bw_\lambda M= V\oplus V',$$ as objects of $\hat{\cal F}$. Then we have $$(\bw_\lambda M)_\lambda= V_\lambda\oplus V'_\lambda,$$ as $\hat{\bu}^0_q$--modules. Since $$\rm{Ann}_\lambda(\bw_\lambda M)_\lambda=0$$ it follows that  $V_\lambda$ and $V'_\lambda$ are $\ba_\lambda$--modules and hence $$M= V_\lambda\oplus V'_\lambda$$ as $\ba_\lambda$--modules. The converse implication is trivial and the proposition is established.
\end{pf}

 \subsection{} \label{locweyl} Given $\bpi\in\cal P^+$, let $\bc_\bpi$ be the one--dimensional representation of $\hat{\bu}_q^0$ defined by taking the quotient by  the maximal ideal $\bi(\bpi)$ generated by the elements $$\{h_{i,r}-h_{i,r}(\bpi): i\in I, r\in\bz^\times \},$$ or equivalently by the elements$$\{\Lambda_{i,r}-\Lambda_{i,r}(\bpi): i\in I, r\in\bz^\times\}.$$ It is clear from \eqref{defrelalambda} that $\bc_\bpi$ is a $\ba_\lambda$--module.
 The local Weyl module $W(\bpi)$ is given by, $$W(\bpi)= \bw_\lambda\bc_\bpi,\ \ w(\bpi)=w_\lambda\otimes 1.$$ An alternative definition of $W(\bpi)$ is that it is the quotient of $W(\lambda)$ obtained by imposing the additional relations: \begin{equation}\label{finitegen} (h_{i,r}-h_{i,r}(\bpi))w_\lambda=0, \qquad{\rm equivalently}\ \ (\Lambda_{i,r}-\Lambda_{i,r}(\bpi))w_\lambda=0, \end{equation} for all $i\in I$, $r\in\bz^\times$.
Clearly $W(\bpi)$ is a highest-$\ell$-weight module with $\ell$--weight $\bpi$ and $W(\bpi)$ is universal with this property. In particular $V(\bpi)$ is a quotient of $W(\bpi)$. Moreover $V(\bpi)$ is the quotient of $W(\bpi)$ by the maximal submodule not containing $W(\bpi)_{\wt\bpi}=\bc w(\bpi)$.

\subsection{} We now prove,
\begin{prop} \label{zeroext} Let $\lambda\in P^+$ and let  $V$ be any $\hat{\bu}_q$--module such that $\wt V\subset \lambda-Q^+$  and assume that $\lambda-(\lambda_i+1)\alpha_i\notin \wt V$. Then,  $$\Ext^1_{\hat{\bu}_q}(W(\lambda), V)=0,$$ or equivalently, any short exact sequence $$0\to V\stackrel{\iota}{\longrightarrow} W\stackrel{\tau}{\longrightarrow} W(\lambda)\to 0$$ of $\hat{\bu}_q$--modules is split.
\end{prop}
\begin{pf} Observe that $\wt W\subset\lambda-Q^+$ and also that $\lambda-(\lambda_i+1)\alpha_i\notin\wt W$. Hence if  $w\in W_\lambda$ is such that $\tau(w)=w_\lambda$, we have $$x_{i,r}^+w=0,\ \ (x_{i,0}^-)^{\lambda_i+1}w=0.$$ It follows there exists a  $\hat{\bu}_q$-module map $\eta:W(\lambda)\to\hat{\bu}_qw$ with $\eta(w_\lambda)=w$. The composite map $\tau.\eta: W(\lambda)\to W(\lambda)$ then satisfies $\tau.\eta(w_\lambda)=w_\lambda$ and hence is the identity map and the Lemma is established.
\end{pf}

\subsection{} One can compute $\Ext^r_{\hat{\bu}_q}(W(\bpi), W(\bpi))$ for all $\bpi\in\cal P^+$ using the Koszul complex for the $\ba_\lambda$--module $\bc_\bpi$ and using Theorem \ref{free}.  Consider the  the special case when $r=1$  in which case we have $$\dim\Ext^1_{\ba_\lambda}(\bc_\bpi,\bc_\bpi)=\sum_{i\in I}\lambda_i.$$
\begin{prop}\label{dimextweyl} Let $\bpi\in\cal P^+$ such that $\wt\bpi=\lambda$. Then $$\Ext^1_{\hat{\cal F}}(W(\bpi), W(\bpi)) \cong  \Ext^1_{\ba_\lambda}(\bc_\bpi,\bc_\bpi),$$ and hence $$\dim\Ext^1_{\hat{\cal F}}(W(\bpi), W(\bpi))= \sum_{i\in I}\lambda_i.$$
\end{prop}
\begin{pf} Let $\eta: W(\lambda)\to W(\bpi)$ be the map of $\hat{\bu}_q$--modules such that $\eta(w_\lambda)=w(\bpi)$. Since $\dim W(\bpi)_\lambda=1$, we have $$\dim\Hom_{\hat{\bu}_q}(W(\bpi), W(\bpi)) \ = 1\ = \dim\Hom_{\hat{\bu}_q}(W(\lambda), W(\bpi)).$$
Applying $\Hom_{\hat{\cal F}}(-,W(\bpi))$ to the short exact sequence $0\to\ker\eta\to W(\lambda)\to W(\bpi) \to 0$,
and using Proposition \ref{zeroext} we get $$\Hom_{\hat{\bu}_q}(\ker\eta, W(\bpi))\cong\Ext^1_{\hat{\bu}_q}(W(\bpi), W(\bpi)).$$ It follows from \eqref{finitegen} that  $\ker\eta$ is generated as a $\hat{\bu}_q$--module by the elements $\{h_{i,r}-h_{i,r}(\bpi) :i\in I, r\in\bz.\}$ Since $h_{i,r}(\bpi)$ is determined by the values of $h_{i,s}(\bpi)$ for $1\le s\le \lambda_i$, we see that
$$\dim\Ext^1_{\hat{\bu}_q}(W(\bpi), W(\bpi))=\dim \Hom_{\hat{\bu}_q}(\ker\eta, W(\bpi))\le \sum_{i\in I}\lambda_i.$$
Finally, this bound must be achieved since Proposition \ref{exact} defines a map
$$\Ext^1_{\ba_\lambda}(\bc_\bpi,\bc_\bpi)\to \Ext^1_{\hat{\cal F}}(\bw_\lambda\bc_\bpi,\bw_\lambda\bc_\bpi)$$ which is injective since the functor exact.
\end{pf}

 \subsection{} The structure of the local Weyl modules is known through the work of \cite{AK}, \cite{Cbraid}, \cite{VV}. To describe the result, we recall the definition of fundamental $\ell$--weights. For $i\in I$ and $a\in\bc^\times$, let $\bpi_{i,a}=(\pi_1,\cdots, \pi_n)\in\cal P^+$ be such that $\pi_i= 1-au$ and $\pi_j=1$ otherwise.
 \begin{prop} Let $\bpi\in\cal P^+$. Then  there exists $k\in\bz_+$, $i_s\in I$, $a_s\in\bc^\times$, $1\le s\le k$, such that $$W(\bpi)\cong V(\bpi_{i_1,a_1})\otimes\cdots\otimes V(\bpi_{i_k,a_k}).$$ Moreover $W(\bpi)$ is irreducible if $a_p/a_\ell\notin q^\bz $ for all $1\le p,\ell\le s$.\hfill\qedsymbol
 \end{prop}

In fact it was shown in \cite{Cbraid} that $W(\bpi)$ is irreducible if $a_p= a_\ell$ for all $1\le p,\ell\le s$. Combining this with Proposition \ref{dimextweyl} we get, that \begin{equation} \dim\Ext_{\hat{\cal F}}(V(\bpi),V(\bpi))=1,\ \ \bpi\ \ {\rm is\ generic}\ \ \implies V(\bpi)\ {\rm{is\ prime}}.\end{equation}

\section{Proof of Theorem 1 ${\rm(iii)}$.}\label{s:tp}
In addition to proving part (iii) of Theorem 1,  we also prove  some results on self--extensions which are needed later in the paper and use the study of global and local Weyl modules.
\subsection{}
\begin{prop}\label{include}  Let $\bpi\in\cal P^+$ and let $V$ be any self--extension of $V(\bpi)$.
 The restriction $V\to V_{\wt\bpi}$ induces an injective map of vector spaces $$\Ext^1_{\hat{\cal F}}(V(\bpi), V(\bpi))\to\Ext^1_{\ba_\lambda}(\bc_\bpi,\bc_\bpi)\cong\Ext^1_{\hat{\cal F}}(W(\bpi), W(\bpi)).$$\end{prop}

    \begin{pf} Let $K(\bpi)$ (resp. $\tilde{K}(\bpi)$) be the kernel of the map $W(\lambda)\to W(\bpi)$ ({\rm resp.} $W(\lambda)\to V(\bpi))$). Clearly $K(\bpi)\subset\tilde K(\bpi)$ and moreover, since $$ W(\bpi)_{\wt\bpi}\cong V(\bpi)_{\wt\bpi}\cong\bc_\bpi,$$ it follows that \begin{equation}\label{homzero}\tilde{K}(\bpi)_{\wt\bpi}=K(\bpi)_{\wt\bpi},\ \ {\rm and}\ \ \Hom_{\hat{\cal F}}(\tilde K(\bpi)/K(\bpi), V(\bpi))=0.\end{equation}

     Using Proposition \ref{zeroext} and the fact that \begin{gather*}\dim \Hom_{\hat{\cal F}}(W(\bpi), W(\bpi)) = \dim \Hom_{\hat{\cal F}}(W(\lambda), W(\bpi))\\ =\dim \Hom_{\hat{\cal F}}(W(\lambda), V(\bpi))=\dim \Hom_{\hat{\cal F}}(V(\bpi), V(\bpi))=1,\end{gather*} we see that \begin{gather}\label{ext1}\Ext^1_{\hat{\cal F}}(W(\bpi), W(\bpi))\cong\Hom_{\hat{\cal F}}(K(\bpi), W(\bpi)),\\ \label{ext2}\Ext^1_{\hat{\cal F}}(V(\bpi), V(\bpi))\cong\Hom_{\hat{\cal F}}(\tilde{K}(\bpi),V(\bpi)).\end{gather}
 Applying $\Hom_{\hat{\cal F}}(-, V(\bpi))$ to the short exact sequence $$0\to K(\bpi)\to\tilde K(\bpi)\to\tilde K(\bpi)/K(\bpi)\to 0,$$ and using \eqref{homzero} we get an inclusion, $$\Hom_{\hat{\cal F}}(\tilde K(\bpi), V(\bpi))\hookrightarrow\Hom_{\hat{\cal F}}(K(\bpi), V(\bpi)).$$ The Proposition now follows from \eqref{ext1}, \eqref{ext2} and  Proposition \ref{dimextweyl}.\end{pf}

\subsection{} We now determine the image of the inclusion given in Proposition \ref{include}. Equivalently, we answer the following question: when does a self--extension of $W(\bpi)$ determine a self--extension of $V(\bpi)$. It is convenient to introduce the following notation: given any $\hat{\bu}_q$--module $W$ with $\wt W\subset\lambda-Q^+$, let $W^\lambda$ be the unique maximal submodule of $W$ such that $$W^\lambda\cap W_\lambda=0.$$ It is clear that $W^\lambda$ exists and is unique -- one just takes the sum of all submodules $U$ of $W$ such that $U_\lambda=0$. Recall that $$W(\bpi)/W(\bpi)^\lambda\cong V(\bpi).$$

 \begin{lem}\label{image} Let $\bpi\in\cal P^+$ and assume that $W$ is a (non--split) self extension of $W(\bpi)$, $$0\to W(\bpi)\stackrel{\iota}\longrightarrow W\stackrel{\tau}\longrightarrow W(\bpi)\to 0.$$ If  $\tau(W^{\wt\bpi})=W(\bpi)^{\wt\bpi}$, there exists a (non--split) self--extension $V$ of $V(\bpi)$ with $$ V_{\wt\bpi}\cong_{\ba_\lambda} W_{\wt\bpi}.$$
 \end{lem}
 \begin{pf} Since $\iota(W(\bpi)^{\wt\bpi})\subset W^{\wt\bpi}$ we see that the restrictions of $\iota$ and $\tau$ give a short exact sequence $$0\to W(\bpi)^{\wt\bpi}\stackrel{\iota}\longrightarrow W^{\wt\bpi}\stackrel{\tau}\longrightarrow W(\bpi)^{\wt\bpi}\to 0.$$ Setting $V=W/W^{\wt\bpi}$ it follows that $V$ is a self--extension of $V(\bpi)$.
 \end{pf}

\subsection{} \begin{prop}  Suppose that  $V$ is a nontrivial self--extension of $V(\bpi)$ and $\wt\bpi=\lambda$. Then $$V\cong \bw_\lambda V_\lambda/(\bw_\lambda V_\lambda)^\lambda.$$\end{prop}
 \begin{pf} If $V$ is a non--trivial self extension, then it follows from Lemma \ref{nontriv} that there exists $v\in V_\lambda$ such that $V=\hat{\bu}_qv$. Since $\wt V\subset\lambda-Q^+$ we see that $V$ is a quotient of $W(\lambda)$ and hence also of $\bw_\lambda V_\lambda$.  Since $$\dim V_\lambda=\dim(\bw_\lambda V_\lambda)_\lambda=2,$$ it follows that $\bw_\lambda V_\lambda/(\bw_\lambda V_\lambda)^\lambda$ must be a quotient of $V$ and hence is either isomorphic to $V$ or to $V(\bpi)$. But the latter is impossible since $\dim V(\bpi)_\lambda=1$ and the proof is complete.\end{pf}

The following is now immediate.
\begin{cor}\label{top} Suppose that $V$ and $V'$ are self--extensions of $V(\bpi)$ and $\wt\bpi=\lambda$. Then $$V\cong_{\hat{\cal F}} V'\iff\ \ V_\lambda\cong_{\ba_\lambda} V'_\lambda.$$
    \end{cor}

\subsection{} The next proposition along with Lemma \ref{lindep} proves part (iii) of Theorem 1.
\begin{prop} Let $V_1$ and $V_2$ be nontrivial self extensions of $V(\bpi)$ for some $\bpi\in\cal P^+$. Then for all $\bpi_1\in\cal P^+$ with $V(\bpi)\otimes V(\bpi_1)$ irreducible, we have $$V_1\otimes V(\bpi_1)\cong V_2\otimes V(\bpi_1)\iff V_1 \cong V_2.$$\end{prop}
\begin{pf}
 Let $\eta:V_1\otimes V(\bpi_1)\to V_2\otimes V(\bpi_1)$ be an isomorphism of $\hat{\bu}_q$--modules. By Lemma \ref{nontriv2} we know that $V_j\otimes V(\bpi_1)$ is a nontrivial extension of $V(\bpi)\otimes V(\bpi_1)$. Let $\iota_j: V(\bpi)\to V_j$ be the inclusion. Since $\iota_1(v(\bpi))\otimes v(\bpi_1)$ is a highest-$\ell$-weight vector in $V_1\otimes V(\bpi_1)$, we may assume without loss of generality that$$\eta(\iota_1(v(\bpi))\otimes v(\bpi_1))=\iota_2(v(\bpi))\otimes v(\bpi_1). $$
 Let $v_1\in (V_1)_{\wt\bpi}$ be linearly independent from $\iota_1(v(\bpi))$. Writing
  \begin{equation}\label{eta}\eta(v_1\otimes v(\bpi_1))=v_2\otimes v(\bpi_1), \end{equation} we see that $v_2$ and $\iota_2(v(\bpi_2))$ are linearly independent elements of $(V_2)_{\wt\bpi}$. Applying $h_{i,r}$ to both side of \eqref{eta}, we get $$\eta(h_{i,r}v_1\otimes v(\bpi_1))= h_{i,r}v_2\otimes v(\bpi_1),$$ for all $i\in I$, $r\in\bz^\times$.  Writing $$h_{i,r}v_j= h_{i,r}(\bpi)v_j+ c^j_{i,r}\iota_j(v(\bpi)), $$ we find now that $c_{i,r}^1=c_{i,r}^2$ for all $i\in I$, $r\in\bz^\times$.  Hence the map $v_1\to v_2$, $\iota_1(v(\bpi))\to\iota_2(v(\bpi))$ defines an isomorphism $$(V_1)_{\wt\bpi}\cong (V_2)_{\wt\bpi}$$ of $\ba_{\wt\bpi}$--modules. By Corollary \ref{top} we see that this implies $V_1\cong V_2$ as $\hat{\bu}_q$--modules. The converse statement is trivial and the proof is complete.
\end{pf}

\section{Proof of Theorem 1 ${\rm(iv)}$.}\label{s:sl2}
Throughout this section we shall be concerned with $R$ being of type $A_1$. In this case, $I=\{1\}$ and for ease of notation, we denote the elements $x_{i, r}^\pm$ as just $x_r^\pm$. Since we will not be using the Chevalley generators in this section there should be no confusion. We also remark that $d_1=1$ and hence we just denote by $[r]$ the quantum number $[r]_1$. Finally, we identify $P$ with $\bz$ and $Q$ with $2\bz$ and denote the modules $W(\lambda)$ by $W(n)$ etc.

\subsection{}\label{actsl2} Given $m\in\bz_+$ and $a\in\bc^\times$, set $$\bpi(m,a)=(1-aq^{m-1}u)(1-aq^{m-3}u)\cdots (1-aq^{-m+1}u).$$ The representation $V(\bpi(m,a))$ has the following explicit realization. It has a basis $v_m,\cdots, v_{0}$ and the action of the generators $x_r^\pm$ is given by,\\
\begin{equation}\label{e:xr-toprest}
x_{r}^+v_j =\left(aq^{-m+2j+2}\right)^r[j+1]v_{j+1} \qquad  x_{r}^-v_j = \left(aq^{-m+2j}\right)^r[m-j+1]v_{j-1},
\end{equation}
where $0\le j\le m$ and we understand that $v_{-1}=v_{m+1}=0$.
 The action of the remaining generators is determined by these and we note for future use that for all $r\in\bz$ with $r\ne 0$, we have
\begin{gather}\label{e:ev(phi_r)}
\phi_{\pm r}^\pm v_m=\pm(q-q^{-1})(aq^m)^{\pm r}[m]v_m,\qquad  h_{r}v_m= a^r\frac{[rm]}{r}v_m.
\end{gather}
Clearly $V(\bpi(m,a))$ is irreducible for the subalgebra $\bu_q$ and hence by Corollary \ref{class} is a prime object of $\hat{\cal F}$.
The following is a consequence of \cite[Theorem 4.8]{CPqa}.
\begin{prop} \label{primeclass} Any prime simple object in $\hat{\cal F}$ is isomorphic to $V(\bpi(m,a))$ for some $m\in\bz_+,a\in\bc^\times$. Moreover, for all $s\in\bz_+$, we have  $$V(\bpi(m,a)^s)\cong V(\bpi(m,a))^{\otimes s}.$$\hfill\qedsymbol
\end{prop}

\subsection{}  The next proposition together with part (iii) of Theorem 1 and  Proposition \ref{primeclass}  establishes part (iv) of  Theorem 1.
\begin{prop}\label{reduction}  Let $m\in\bz_+$, $a\in\bc^\times$. Then,
\begin{enumerit}\item[(i)] $\dim\Ext^1_{\hat{\cal F}}(V(\bpi(m,a)),V(\bpi(m,a)))=1,$ \item[(ii)] $\dim\Ext^1_{\hat{\cal F}}(V(\bpi(m,a))^{\otimes 2},V(\bpi(m,a))^{\otimes 2})\ge 2.$
\end{enumerit}
\end{prop}

The rest of this section is devoted to proving the proposition.

\subsection{} Recall from Section 4 that the local Weyl module $W(\bpi(m,a))$ is the module generated by an element $w_m$ with relations: $$x_r^+ w_m=0,\ \  \ \ h_r w_m=a^r\frac{[rm]}{r}w_m\ \ ,\ \ (x_0^-)^{m+1}w_m=0.$$
\begin{prop} \label{genrel} Let $m\in\bz_+$, $a\in\bc^\times$.
\begin{enumerit}
 \item[(i)] The module $V(\bpi(m,a))$ is the quotient of $W(\bpi(m,a))$ obtained by imposing the single additional relation $$(x_1^--aq^mx_0^-)w_m=0.$$
     \item[(ii)] The module $V(\bpi(m,a))^{\otimes 2}$ is the quotient of $W(\bpi(m,a)^2)$ obtained by imposing the single additional relation $$(x^-_2-2aq^mx^-_1+a^2q^{2m}x^-_0)w_{2m}=0.$$\end{enumerit}\end{prop}
         \begin{pf} To prove (i), notice that the formulae given in Section \ref{actsl2} imply that  the element $v_m\in V(\bpi(m,a))$ satisfies, $$(x_1^--aq^mx_0^-)v_m=0,$$  In particular, if $$W =W(\bpi(m,a))/\hat{\bu}_q(x_1^--aq^mx_0^-)w_m,$$ then $V(\bpi(m,a))$ is a quotient of $W$. Part (i) follows if we prove that $\dim W\le m+1$. For this we denote by $\bar w_m$ the image of $w_m$ in $W$ and  observe that $W$ is spanned by  $\bar w_m$ and elements of the form $x_{s_1}^-\cdots x^-_{s_k}\bar w_m$ where $1\le k\le m$. Since
$$0=
[h_r, x_1^--aq^mx_0^-]\bar w_m = \frac{[2r]}{r}(x^-_{r+1}-aq^m x^-_r)\bar w_m,$$  we get that
 $x_r^-\bar w_m \in \bc x_0^- \bar w_m$
for all $r\in\bz$.  This implies that $W$ is spanned by elements of the form $x_{s_1}^-\cdots x_{s_{k-1}}^- x^-_{0}\bar w_m$ where $1\le k\le m$.  Suppose that we have proved that we may take $s_2=\cdots=s_k=0$. If $s_1=\pm 1$, then using the relation $$x_0^-x_{\pm 1}^-=q^{\pm 2}x_{\pm 1}^-x_0^-$$ shows that $x_{\pm 1}^-(x_0^-)^{k-1}\bar w_m$ is a multiple of $(x_0^-)^{k}\bar w_m$. An obvious induction on $s$ using the relation $$x_s^-x_0^--q^{-2}x_0^-x_s^-= q^{-2}x_1^-x_{s-1}^-- x_{s-1}^-x_1^-,$$ now proves that $W$ is spanned by elements of the form $(x_0^-)^s\bar w_m$, $1\le s\le m$ and hence $\dim W\le m+1$ as required.

         The proof of part (ii) is very similar. We  observe that we have the relation $$(x^-_2-2aq^mx^-_1+a^22q^{2m}x^-_0)(v_m\otimes v_m)=0,$$ in $V(\bpi(m,a))^{\otimes 2}$. We set $$W= W(\bpi(m,a)^2)/\hat{\bu}_q(x^-_2-2aq^mx^-_1+a^22q^{2m}x^-_0) w_{2m},$$  and let $\bar w_{2m}$ be the image of $w_{2m}$ in $W$. We now prove  exactly as before that $W$ is spanned by $\bar w_{2m}$ and  elements of the form $(x_1^-)^s(x_0^-)^\ell\bar w_m$ with $1\le s+\ell\le 2m$. The spanning set is now of cardinality bigger than $(m+1)^2$ if $m>1$. To show that in fact we can choose a suitable subset of cardinality at most $(m+1)^2$ we observe that $$\dim W_{2r}=\dim W_{-2r},$$ and hence it is enough to determine a bound for $\dim W_{2r}$ for $0\le r\le m$. This bound is easily seen to be $m-r+1$ and so we now have $$\dim W= 2\dim W_{2m} +2\dim W_{2m-2}+\cdots +2\dim W_2 + \dim W_0= 2(1+2\cdots+ m)+ (m+1) =(m+1)^2.$$ This completes the proof of the Proposition.
         \end{pf}

         \subsection{} We now prove Proposition \ref{reduction}(i). Consider the canonical map from the global Weyl module $\eta: W(m)\to V(\bpi(m,a))$ which sends $w_m\to v_m$. We claim that $\ker\eta$ is generated by the element $v= (x_1^--aq^mx_0^-)w_m$.  By Proposition \ref{genrel}(i) we see that $v\in\ker\eta$. Moreover,
$$(q-q^{-1}) x_r^+ v= \begin{cases}(\phi^\pm_{r+1}-aq^m\phi^\pm_{r})w_m,\ \ r\in\bz, r\ne 0,-1,\\ (\phi^+_1-aq^m(q^m-q^{-m}))w_m,\ \ r=0,\\
        ( (q^m-q^{-m})-aq^m\phi_{-1}^-)w_m,\ \ r=-1.
         \end{cases}  $$
An induction on $r$  shows that  $$(\phi^\pm_r-(aq^m)^r(q^m-q^{-m}))w_m\in\hat{\bu}_qv,\ \ r\ne 0. $$  Setting $$\tilde W=W(m)/\hat{\bu}_qv,$$ we see that the defining relations of $W(\bpi(m,a))$ imply that $\tilde W$ is a quotient of $W(\bpi(m,a))$. It now follows from Proposition \ref{genrel}(i) that $$\tilde W\cong V(\bpi(m,a)),$$ and the claim is established.

By Proposition \ref{zeroext}, $\Ext^1_{\hat{\cal F}}(W(m),V(\bpi(m,a)))=0$. Thus,
applying $\Hom_{\hat{\cal F}}(-,V(\bpi(m,a)))$ to the short exact sequence $$ 0\to \ker \eta \to W(m) \to V(\bpi(m,a)) \to 0,$$ and noting also that $\Hom_{\hat{\cal F}}(W(m),V(\bpi(m,a)))\cong \Hom_{\hat{\cal F}}(V(\bpi(m,a)),V(\bpi(m,a)))$, one finds
$$\dim\Ext^1_{\hat{\cal F}}(V(\bpi(m,a)), V(\bpi(m,a)))=\dim\Hom_{\hat{\cal F}}(\ker\eta, V(\bpi(m,a)))\le 1.$$ By Proposition \ref{thm1i} we know $\be(V(\bpi(m,a))$ is a non--trivial self extension and hence part (i) of Proposition \ref{reduction} is proved.

         \subsection{} The proof of part (ii) proceeds as follows. We construct an ideal of $\ba_{2m}$ of codimension two and show that it can be used to define a non--trivial self--extension of $W(\bpi)$. We then show that this self--extension satisfies the conditions of Lemma  \ref{image} and hence defines a non--trivial self--extension $V$  of $V(\bpi)$. Finally, we prove that this extension is not isomorphic to $\be(V(\bpi))$. Proposition \ref{lindep} implies that  $[V]$ and $[\be(V(\bpi))
         ]$ are linearly independent elements of $\Ext^1_{\hat{\cal F}_q}(V(\bpi), V(\bpi))$ which proves (ii).

        \subsection{}  We recall for the reader's convenience that $$\ba_{2m}=\bc[\bar\Lambda_1,\bar\Lambda_2, \cdots,\bar\Lambda_{2m},\bar\Lambda_{2m}^{-1}],$$ and that we have an algebra homomorphism $\hat{\bu}_q^0\to\ba_{2m}$ given by $$\Lambda_r\to\begin{cases} 0,\  \ |r|\ge 2m+1,\\ \bar\Lambda_r, \ \ 0< r\le 2m,\\ \bar\Lambda_{2m+r}\bar\Lambda_{2m}^{-1},\ \ -2m\le r \le 0.\end{cases}$$
         Let  $\bi$ be the ideal in $\ba_{2m}$ generated by $(\bar\Lambda_{1} +2a[m] )^2$, and the elements: \begin{equation}\label{defineideal} [r+2]\bar\Lambda_{r+2}-\left(q^{r+1}\bar\Lambda_{1}+2aq^m[r+1]\right)\bar\Lambda_{r+1}-a^2[2m-r]\bar\Lambda_r,\ \ 0\le r\le 2m.\end{equation} Set $\bpi=\bpi(m,a)^2$.
          \begin{lem} The ideal $\bi$  is of  codimension two and we have a non--split short exact sequence of $\ba_{2m}$--modules, $$0\to\bc(\bpi)\to \ba_{2m}/\bi\to\bc(\bpi)\to 0.$$\end{lem}
         \begin{pf}  We  first prove that $\bi\subset\bi(\bpi)$ so that we have a surjective map $\ba_{2m}/\bi\to\bc(\bpi)\to 0$ of $\ba_{2m}$--modules. Write $\bpi=\sum_{s=0}^{2m}d_su^s$, and using the fact  that \begin{equation}\label{finally}(1-aq^{-m}u)^2\bpi(qu)= (1-aq^mu)^2\bpi(q^{-1}u),\end{equation}   we see that \eqref{defineideal} is identically satisfied if we replace $\bar\Lambda_r$ by $d_r$ and hence $\bi\subset\bi(\bpi)$.
         Next, note that  after an obvious change of variables of the form
         $$\bar\Lambda_{r}\to X_r=\bar\Lambda_r+ p_r(\bar\Lambda_{r-1},\cdots,\bar\Lambda_1),\qquad r>1,\ \ \bar\Lambda_1=X_1,$$ we have that $\bi$  is generated by $$(X_1+2a[m] )^2,\ \ X_2,\cdots, X_{2m}.$$  Hence the ideal generated by these elements in $\bc[X_1,\cdots, X_{2m}]=\bc[\bar\Lambda_1,\cdots, \bar\Lambda_{2m}]$ is  of codimension two.
         Since $\bar\Lambda_{2m}\notin\bi$ (recall $\bar\Lambda_{2m}\notin\bi(\bpi))$ the conclusion  does not change if we localize at $\Lambda_{2m}$ and work with the ideal $\bi$. In particular we have proved that $\ba_{2m}/\bi$ is an indecomposable module of dimension two and that  we have a non--split short exact sequence $$0\to (\bar\Lambda_1+2a[m])\ba_{2m}/\bi\to\ba_{2m}/\bi\to\bc(\bpi)\to 0,$$ of $\ba_{2m}$--modules, or equivalently, we have a non--split short $$0\to\bc_\bpi\to \ba_{2m}/\bi\to\bc_\bpi\to 0.$$\end{pf}
         \subsection{}
         We now set,
          $$W=W(2m)\otimes_{\ba_{2m}} \ba_{2m}/\bi,\qquad w= w_{2m}\otimes 1,$$ and observe that by Theorem \ref{free} there exists a   a non--split short exact sequence $$0\to W(\bpi)\stackrel{\iota}\longrightarrow W\stackrel{\tau}\longrightarrow W(\bpi)\to 0.$$
          Recall that $W^{2m}$ is    the unique  maximal submodule of $W$ such that $W_{2m}\cap W^{2m}=0$,
           and let $$\tilde w =(x_2^--2aq^mx_1^-+a^2q^{2m}x_0^-)w.$$
           \begin{lem} We have  $$\hat{\bu}_q\tilde w\subset W^{2m},$$ and hence $\tau: W^{2m}\to W(\bpi)^{2m}$ is surjective.\end{lem}
          \begin{pf} The subspace $(\hat{\bu}_q\tilde w)_{2m}$ is the $\hat{\bu}_q^0$--submodule generated by the elements $x_r^+\tilde w$, $r\in\bz$ and hence it suffices to prove that $x_r^+\tilde w=0$.
          This means we must prove that \begin{equation}\label{philambda}0=x_r^+\tilde w=\begin{cases}(\phi_{r+2}^\pm -2aq^m\phi_{r+1}^\pm+a^2q^{2m}\phi_r^\pm) w,\ \ r\ne -2,-1,0\\ \\
         (\phi_{2}^\pm -2aq^m\phi_{1}^\pm+a^2q^{2m}(q^{2m}-q^{-2m})) w,\ \ r=0,\\ \\
         (\phi_{1}^\pm -2aq^m(q^{2m}-q^{-2m})+a^2q^{2m}\phi_{r-1}^\pm) w,\ \ r=-1,\\ \\
         ((q^{2m}-q^{-2m}) -2aq^m\phi_{-1}^\pm+a^2q^{2m}\phi_{-2}^\pm) w,\ \ r=-2.\end{cases}\end{equation}
          Using the functional equation,
         $$\phi^\pm(u)=\frac{\Lambda^\pm(q^{\mp 1}u)}{\Lambda^\pm(q^{\pm 1}u)},$$ we see that  \eqref{philambda} is equivalent to requiring,\begin{gather}\label{red2}
(1-aq^mu)^2\Lambda^+(q^{-1}u)w= \left((1-aq^{m}u)^2 - a^2(q^{2m}-q^{-2m})u^2 - (q-q^{-1})\Lambda_{1}u\right)\Lambda^+(qu)w,\\ \nn \\
\label{red3} (u-aq^m)^2\Lambda^-(q^{-1}u)w=\left((u-aq^m)^2+(q^{4m}-1)u^2+a^2q^{2m}(q-q^{-1})\Lambda_{-1}u\right)\Lambda^-(qu)w,\\ \nn \\
\label{red4}(q^{2m}\Lambda_1+2a[2m]q^m+a^2\Lambda_{-1})w=0.
\end{gather}

Since $W$ is a quotient of $W(2m)$ we have that ${\rm Ann}_{2m}w=0$ and hence it suffices to prove that the equations in \eqref{red2}, \eqref{red3},\eqref{red4} are satisfied in $\ba_{2m}$.
 It is easily seen that \eqref{red2} is  exactly \eqref{defineideal}. To see that the other two equations are satisfied, one recalls that we have the relation $$\Lambda_{-2m}\Lambda_r=\Lambda_{-2m+r},\ \ 0\le r\le 2m.$$ Then \eqref{red4} follows by taking the case of $r=2m-1$ in \eqref{defineideal}, which gives $$(q^{2m}\Lambda_1+2aq^m[2m])\Lambda_{2m}+a^2\Lambda_{2m-1}=0.$$
 Multiplying through by $\Lambda_{2m}^{-1}$ gives the result. Equation \eqref{red3} follows similarly by using the cases when $0\le r\le 2m-2$.\end{pf}

\subsection{} As a consequence of the preceding Lemma and Lemma \ref{image}    we have a non-split  short exact sequence $$0\to V(\bpi)\to W/W^{2m}\to V(\bpi)\to 0.$$
The final step is to show that this extension is not isomorphic to  $\be(V(\bpi))$. For this, we observe that if $\eta: W/W^{2m}\to\be(V(\bpi))$ is an isomorphism, then we must have $\eta(\bar w)=(c_1v(\bpi), c_2v(\bpi))$ for some $c_1\ne 0$, where $\bar w$ is the image of $w$ in $W/W^{2m}$. Since $$\eta(\phi^+_2-2aq^m\phi^+_1+a^2q^{2m}(q^{2m}-q^{-2m}))(v(\bpi),0)=(0, (2\phi^+_2-2aq^m\phi_1^+)v(\bpi))\ne 0,$$ and $\eta(\phi^+_2-2aq^m\phi^+_1+a^2q^{2m}(q^{2m}-q^{-2m}))\bar w=0$, we have a contradiction. This completes the proof of Theorem 1 (iv).

\section{Proof of Theorem \ref{thm2}}
\newcommand{\gb}{\boldsymbol}
In this section we prove Theorem \ref{thm2}. We begin by noting some additional consequences of the results of the preceding sections.
\subsection{} Given a connected subset $J$ of $I$ let  $\hat{\bu}_q^J$ be the subalgebra of $\hat\bu_q$ generated by the elements  $x^\pm_{i,r}$, $h_{i,s}$, $k_i^{\pm 1}$,  $i\in J$, $r\in\bz$, $s\in\bz^\times$. If $R_J$ is the subset of the root system spanned by the elements $\alpha_j$, $j\in J$, then $\hat{\bu}_q^J$ is the quantum loop algebra associated to $R_J$ with parameter $q_J$ where $q_J= q^{\rm{min}\{d_j: j\in J\}}$. In the special case when $J=\{i\}$ we write $\hat{\bu}_q^i$ for the algebra $\hat{\bu}_q^J$ and note that $\hat{\bu}_q^i$ is the quantum loop algebra associated to $A_1$ with parameter $q_i$.
Let $\cal P_J^+$ be the submonoid of $\cal P^+$ consisting of $I$--tuples $\bpi=(\pi_1,\cdots,\pi_n)$ satisfying $\pi_i=1$ if $i\notin J$.  It is also convenient to regard $\cal P_J^+$ as a quotient of $\cal P^+$ via the map which sends $$\bpi=(\pi_i)_{i\in I}\to \bpi_J=(\pi_j)_{j\in J}.$$
The category $\hat{\cal F}^J$ is defined in the obvious way and the elements of $\cal P_J^+$ index the isomorphism classes of the simple objects of $\hat{\cal F}_J$. %
  The following is easily established.
\begin{lem}\label{restrict1} Let  $\bpi\in\cal P^+$. The $\hat{\bu}_q^J$--submodule of $V(\bpi)$ generated by $v(\bpi)$ is isomorphic to $V(\bpi_J)$. \hfill\qedsymbol\end{lem}

\subsection{} \begin{prop} Let $J$ be a connected subset of $I$.  There exists a canonical map of vector spaces$$\Ext^1_{\hat{\cal F}}(V(\bpi), V(\bpi))\to \Ext^1_{\hat{\cal F}_{J}}(V(\bpi_{J}), V(\bpi_{J})),\ \ [V]\to [V_J]. $$ Moreover $$[V_J]=0\ \ \iff (h_{j,r}-h_{j,r}(\bpi))v=0, \ \ v\in(V_J)_{\wt\bpi},\ \  \ j\in J, \ \ r\in\bz^\times.$$
 \end{prop}
\begin{pf}  Let  $V$ be  a non--trivial self extension of $V(\bpi)$. By Lemma \ref{nontriv} we may choose $v\in V_{\wt\bpi}$ such that $V=\hat{\bu}_qv$. Setting $V_J=\hat{\bu}_q^Jv$, it is clear from Lemma \ref{restrict1} that $V_J$ is a self--extension of $V(\bpi_J)$. If $V$ is the trivial extension, then we set $V_J=V(\bpi_J)\oplus V(\bpi_J)$.  It is now easily checked that $[V]\to [V_J]$ is well defined map of vector spaces. The second statement of the proposition is immediate from Lemma \ref{nontriv}.
 \end{pf}
 The following is immediate.
\begin{cor}\label{restrict2} Let $\bpi\in\cal P^+$ and let $J_1,\cdots, J_m$ be a family of disjoint connected subsets of $I$ such that $I=J_1\cup\cdots\cup J_m$. We have an injective  map of vector spaces, $$\Ext^1_{\hat{\cal F}}(V(\bpi), V(\bpi))\to \bigoplus_{s=1}^m \Ext^1_{\hat{\cal F}_{J_s}}(V(\bpi_{J_s}), V(\bpi_{J_s})).$$
In particular,
$$\dim\Ext^1_{\hat{\cal F}}(V(\bpi), V(\bpi))\le\sum_{i=1}^n\dim\Ext^1_{\hat{\cal F}^i}(V(\bpi_{\{i\}}), V(\bpi_{\{i\}})).$$\hfill\qedsymbol
\end{cor}
\begin{pf} The only statement that requires explanation is that the map $$\Ext^1_{\hat{\cal F}}(V(\bpi), V(\bpi))\to \bigoplus_{s=1}^m \Ext^1_{\hat{\cal F}_{J_s}}(V(\bpi_{J_s}), V(\bpi_{J_s})),$$ is injective. Let  $V$ be  a non--trivial extension of $V(\bpi)$ and recall that we may choose $v\in V_{\wt\bpi}$ with $V=\hat{\bu}_qv$. Moreover  there exists $i\in I$ and $r\in\bz^\times$  such that $h_{i,r}v\notin\bc v$. Choose $1\le s\le m$ such that $i\in J_s$. Then $V_{J_S}$ is a non--trivial extension of $\hat{\bu}_q^J$ and hence $$\Ext^1_{\hat{\cal F}_{J_s}}(V(\bpi_{J_s}), V(\bpi_{J_s}))\ne 0,$$ as  required.
\end{pf}

\subsection{} For  $\bpi\in\cal P^+$, set $$\supp\bpi=\{i\in I: \pi_i\ne 1\}.$$ Together with Proposition \ref{reduction} we have now established the following.
\begin{prop} Let $\bpi=(\pi_1,\cdots,\pi_n)\in\cal P^+$ be such that $$\pi_i=(1-q_i^{m_i-1}a_iu)(1-q_i^{m_i-3}a_iu)\cdots(1-q_i^{-m_i+1}a_iu),$$ for some $m_i\in\bz_+$ and $a_i\in\bc^\times$, $1\le i\le n$. Then $$\dim\Ext^1_{\hat{\cal F}}(V(\bpi), V(\bpi))\le |\supp\bpi|.$$ In particular the space of self extensions of the Kirillov--Reshethikhin modules is one--dimensional. \hfill\qedsymbol
\end{prop}

\subsection{}
We shall prove the following proposition in the rest of the section. \begin{prop}\label{minextreme} Suppose that $\bpi=(\pi_1,\cdots,\pi_n)\in\cal P^+$ is such that $\supp\bpi=\{1,n\}$  and that $$\pi_i=(1-q_i^{m_i-1}c_iu)(1-q_i^{m_i-3}c_iu)\cdots(1-q_i^{-m_i+1}c_iu)\ \ i=1,n,$$and $c_1,c_n\in\bc^\times$. If $$\frac{c_1}{ c_n}=  q^{\pm N},\qquad   N=d_1m_1 - \frac12\sum_{\substack{i\ne j\in [1,n]}}d_ia_{i,j}+d_nm_n$$
then $$\dim\Ext^1_{\hat{\cal F}}(V(\bpi), V(\bpi))=1.$$
\end{prop}

\subsection{} Assuming Proposition \ref{minextreme} the proof of the theorem is completed as follows. Given $\bpi\in\cal P^+$, let $\ba_{\deg\pi_i}$ be the subalgebra of $\ba_{\wt\bpi}$ generated by the elements $\{\Lambda_{i,r}:1\le r\le \lambda_i\}$ and $\Lambda_{i,r}^{-1}$.
It is clear that $$\ba_{\lambda}\cong\bigotimes_{i\in I}\ba_{\deg\pi_i},$$ where $\lambda=\sum_{i\in I}\deg\pi_i\omega_i$. Moreover if $J$ is any connected subdiagram of $I$, and $\lambda_J=\sum_{j\in J}\deg \pi_j\omega_j$, then $$\ba_{\lambda_J}\cong\bigotimes_{j\in J }\ba_{\deg\pi_j}.$$

Suppose that $V$ is a non--split self extension of $V(\bpi)$ so that $$0\to V(\bpi)\stackrel{\iota}\longrightarrow V\stackrel{\tau}\longrightarrow V(\bpi)\to 0.$$ We shall prove that $V\cong\be(V(\bpi))$. We first prove that: if $0\ne v\in V_\lambda$ is such that for some $i\in \supp\bpi$ we have $$h_{i,r}v=h_{i,r}(\bpi) v,\ \ r\in\bz^\times,$$ then $v=a\iota(v(\bpi))$ for some $a\in\bc^\times$. For this, suppose that $i=i_s\in\supp\bpi$ and consider $J= [i_s,i_{s+1}]$ or $J'= [i_{s-1},i_s]$.  If $v\ne a\iota(v(\bpi))$, then $\be(V(\bpi_J))$ and $V_J$  are non-isomorphic extensions of $V(\bpi)$. It follows from Proposition \ref{minextreme} that $[V_J]=0$ and hence $(V_J)_{\lambda_J}$ is an eigenspace for the action of $h_{i_{s+1},r}$ as well. A similar argument works for $J'$ and hence we find that $V_\lambda$ is an eigenspace for $h_{i,r}$ for all $i\in I$, $r\in\bz^\times$ contradicting Lemma \ref{nontriv}.

To prove that $V\cong\be(V(\bpi))$ we must prove that $V_\lambda\cong\be(V(\bpi))_\lambda$ as $\ba_\lambda$--modules. For this, note that as modules for $\ba_{\deg\pi_{i_1}}$ we may assume that there exists a basis $v_1,v_2$ of $V_\lambda$ such that $v_1\to (v(\bpi),0)$, $v_2\to (0,v(\bpi))$ is an isomorphism. Suppose that this is not an isomorphism of $\ba_{\deg\bpi_{i_2}}$--modules. Then $V_\lambda$ and $\be(V(\bpi))_\lambda$ are not isomorphic as $\ba_J$--modules where $J=[i_1,i_2]$ which then implies that $V_J$ and $\be(V(\bpi_J))$ are not isomorphic as $\hat{\bu}_q^J$--modules. Since both extensions are non--trivial this contradicts Proposition \ref{minextreme}. Iterating the argument gives the result that $V_\lambda\cong\be(V(\bpi))_\lambda$ as $\ba_\lambda$--modules and the proof of Theorem 2 is complete once we establish Proposition \ref{minextreme}.

\subsection{}\label{ss:stepdownsimple} Notice that when $\bpi$ is as in Proposition \ref{minextreme},
then  $[1,n]$ is not of type $D$ or $E$. We assume from now that $a_{i,j}\ne 0$ only if  $i=j\pm 1$ and
  also without loss of generality that,
 $$c_n=c_1q^{N},\qquad \ N=d_1m_1 - \frac12\sum_{\substack{i\ne j\in [1,n]}}d_ia_{i,j}+d_nm_n = d_1m_1+d_nm_n - \sum_{i=1}^{n-1} d_ia_{i,i+1}.$$  For  $0\le i\le n$, define elements $w_i\in V(\bpi)$  recursively, by
\begin{equation}\label{ellwt}w_0= v(\bpi),\qquad w_i=x_{i,0}^-w_{i-1}. \end{equation}
 \begin{prop} For $0\le i\le n-1$, we have an isomorphism of $\hat{\bu}_q^i$--modules
\begin{equation}\hat{\bu}_q^{i+1}w_{i}\cong \begin{cases} V(\pi(m_1,c_1)),\ \ i=0,\\
V(\pi(-a_{i+1,i},c_{i+1})),\ \ \ 1\le i\le n-2,\\
V(\pi(m_n-a_{n,n-1},  c_nq_n^{a_{n,n-1}})),\ \ i=n-1.\end{cases}\end{equation}
where $c_2=c_1q_1^{m_1}$ and $c_i=c_{i-1}q_{i-1}^{-a_{i-1,i-2}}$ for $2<i<n$.
\end{prop}
\begin{pf} To prove the proposition we see from Proposition \ref{genrel} that we must show the following:
\begin{enumerit}\item[(i)] $x_{i+1,r}^+w_i=0$,
\item[(ii)]\begin{gather}\label{eigenvw1} h_{1,r}w_{0}= \frac{[rm_1]_{\scriptsize 1}}{r}c_1^rw_0,\ \qquad
 h_{n,r}w_{n-1}= q_n^{ra_{n,n-1}}\frac{[r(m_n-a_{n,n-1})]_n}{r}c_n^rw_{n-1}\\
\label{eigenvw2}  h_{i,r}w_{i-1}= -\frac{[ra_{i,i-1}]_i}{r}c_{i}^rw_{i-1}, \ \ 1<i<n,
 \end{gather}
\item[(iii)]  \begin{gather} (x_{1,1}^--c_1q_1^{m_1}x_{1,0}^-)w_0=0,\ \ (x_{n,1}^--c_nq_n^{m_n}x_{n,0}^-)w_{n-1}=0,\\ (x_{j,1}^--c_{j}q_j^{-a_{j,j-1}}x_{j,0}^-)w_{j-1}=0, 1<j<n.
     \end{gather}
    \end{enumerit}
    Part (i) is trivial. We prove (ii) and (iii) simultaneously by an induction on $i$. Notice that induction begins at $i=1$ by Lemma \ref{nontriv} and Proposition \ref{genrel}.  Suppose that we have proved (ii) for $1<i\le n$ and (iii) for $1<i-1<n$.  We prove (iii) for $i<n$ by showing that\begin{equation}\label{iii} x_{j,r}^+(x_{i,1}^--c_{i}q_i^{-a_{i,i-1}}x_{i,0}^-)w_{i-1}=0,\  \ j\in I, r\in\bz_+.\end{equation}
    For this, writing $\wt\bpi=m_1\omega_1+m_n\omega_n$, we have \begin{equation} \label{jnot1n} x_{j,r}^-w_0=0,\ \ j\ne 1,n,\ \ r\in\bz.\end{equation}
We claim that \begin{equation}\label{hw} x_{j,r}^+w_i=0,\  \  j\ne i,\ \ (i,j)\ne (n,n-1),\ \ r\in\bz^\times.
\end{equation}Since $[x_{j,r}^+, x_{k,0}^-]=0$ if $j\ne k$, it is clear that the claim is true if $j>i$. If $j<i$, then we see that $$x_{j,r}^+w_i= x_{i,0}^-\cdots x_{j+1,0}^-\left(\frac{\phi_{j,r}^\pm}{q-q^{-1}}\right) x_{j-1,0}^-\cdots x_{1,0}^-w.$$ It is easily seen from the defining relations of $\hat{\bu}_q$ that  $$[x_{j+1,0}^-,\phi^\pm_{j,r}]\in\sum_{s\in\bz}\hat{\bu}_q^-x^-_{j+1,s},\qquad [x_{j+1,s}^-,  x_{\ell,0}^-]=0,\ \ \ell<j+1,$$ and \eqref{hw} follows  now by using \eqref{jnot1n}.
It is now clear that \eqref{iii} follows from \eqref{hw}  for $i\ne j,j-1$. For $i=j$ it holds from \eqref{eigenvw2} and  \eqref{e:ev(phi_r)}. If $i=j-1$, then we use the same argument as the one given for establishing \eqref{hw} to see that $$x_{j-1,r}^+(x_{j,1}^--c_{j}q_j^{-a_{j,j-1}}x_{j,0}^-)w_{j-1}\in\hat{\bu}_q(x_{j,1}^-
 -c_{j}q_j^{-a_{j,j-1}}x_{j,0}^-)w_0=0.$$
 This completes the proof of (iii) with $i<n$. For $i=n$ one proves using similar arguments that
\begin{equation}\label{iiin} x_{j,r}^+(x_{n,1}^--c_nq_n^{m_n}x_{n,0}^-)w_{n-1}=0,\  \ j\in I, r\in\bz_+.\end{equation}
We omit the details.

It remains to prove that (ii) holds for $i+1\le n$. If $i+1<n$, then we use
$$h_{j+1,r}w_j=-\frac{[ra_{j+1,j}]_{j+1}}{r} x_{j,r}^-w_{j-1}=-\frac{[ra_{j+1,j}]_{j+1}}{r}(c_{j}q_j^{-a_{j,j-1}})^rw_j, \ \ 1\le j<n-1,$$
while if $i+1=n$
\begin{align*}
h_{n,r}w_{n-1} &=(-\frac{[ra_{n,n-1}]_{n}}{r} x_{n-1,r}^- + x_{n-1,0}^-h_{n,r})w_{n-2}\\
 &=(-\frac{[ra_{n,n-1}]_{n}}{r}(c_{n-1}q_{n-1}^{-a_{n-1,n-2}})^r + \frac{[rm_n]_n}{r}c_n^r) w_{n-1}\\
 &=(-\frac{[ra_{n,n-1}]_{n}}{r}(c_nq_{n}^{a_{n,n-1}}q_n^{-m_n})^r + \frac{[rm_n]_n}{r}c_n^r) w_{n-1}
\end{align*}
as required.
 \end{pf}

\subsection{}
Suppose that  $V$ is a nontrivial self-extension of $V(\bpi)$:
\begin{equation}0  \to V(\bpi) \overset \iota\to V\overset \tau \to V(\bpi) \to 0\nn.\end{equation} To prove Proposition \ref{minextreme} we must show that we have an isomorphism of $\hat{\bu}_q$--modules $$V\cong \be(V(\bpi)).$$ It is enough by Corollary \ref{top} to prove that $$V_{\wt\bpi}\cong \be(V(\bpi))_{\wt\bpi},$$ as modules for $\ba_{\wt\bpi}$ and in fact it is enough to prove that they are isomorphic as $\hat{\bu}_q^0$-modules.
The proposition is a consequence of the following Lemma. Once the Lemma is proved it is clear that the map $\be(V(\bpi))_{\wt\bpi}\to V_{\bpi}$ sending $(v(\bpi),0)\to \tilde w_0$ is an isomorphism of $\hat{\bu}_q^0$--modules.
\begin{lem}\label{walkprop} There exists a basis $\tilde w_0, w_0$ of $V_{\wt\bpi}$ such that
$$h_{i,r}\tilde w_0= h_{i,r}(\bpi) (\tilde w_0 + r w_0),\ \quad h_{i,r}w_0=h_{i,r}(\bpi) w_0, \ \  \ i\in I, r\in\bz^\times.$$
\end{lem}

\subsection{}\label{remark}  We shall use the following remark repeatedly in the proof of the Lemma. It is a special case of results proved elsewhere in this paper, we formulate it here in the precise form that it is used in the proof of the Lemma.
 \begin{rem} Suppose that $\bpi\in\cal P^+$ is such that $\dim\Ext^1_{\hat{\cal F}}(V(\bpi),V(\bpi))=1$ and let $V$ be  any self--extension of $V(\bpi)$.
  Suppose that $\tilde v,v$ is a basis of $V_{\wt\bpi}$ such that $(h_{i,r}-h_{i,r}(\bpi))v=0$ for all $i\in I$ and $r\in\bz^\times$. Then, there exists $z\in\bc^\times$ such that $$h_{i,r}\tilde v= h_{i,r}(\bpi)(\tilde v+zrv).$$ Moreover, this implies that, if $V=V_1\oplus V_2$ is a decomposition of $V$ as a direct sum of $\bu_q$-submodules isomorphic to $V(\bpi)$ with $v\in V_1$ and $\tilde v\in V_2$, the projection of $x_{i,r}^-\tilde v$ onto $V_1$ is $rzx_{i,r}^-v$.
  Finally, $V$ is nontrivial iff and only if $z\ne 0$.
\end{rem}

\subsection{}{\em Proof of Lemma \ref{walkprop}} Let $\tilde w_0$ be such that $\tau(\tilde w_0)=v(\bpi)$ and set $w_0=\iota(v(\bpi))$. If $j\ne 1,n$, then $\hat{\bu}_q^j\tilde w_0$ is the trivial representation of $\hat{\bu}_q^j$. If $j=1,n$, then by Proposition \ref{restrict2} and the results of Section 6
 we know that either  $$\hat{\bu}_q^j\tilde w_0\cong V(\pi(m_j,c_j)),$$ or $$\hat{\bu}_q^j\tilde w_0\cong\be(V(\pi(m_j,c_j)).$$
 In any case, since $\iota(v(\bpi))$ is a joint eigenvector for $h_{i,r}$, $i\in I$, $r\in\bz^\times$, it follows from Section \ref{remark} that there exists $z_1,z_n\in\bc^\times$ such that
$$h_{j,r}\tilde w_0= h_{j,r}(\bpi) (\tilde w_0 + rz_j w_0),\ \qquad\ x_{j,r}^-\tilde w_0 = (c_jq_j^{m_j})^rx_{j,0}^-(\tilde w_0+rz_jw_0),\ \ \ j=1,n. $$
Our goal is to prove that we must have $z_1=z_n$. Since $V$ is non--split this means that we can assume $z_1=z_n=1$  which would establish the Lemma.

 For $1\le i\le n$, define elements $$\tilde w_i=x_{i,0}^-\tilde w_{i-1},\ \  w_i= x_{i,0}^- w_{i-1}. $$
We now prove by induction on $1\le i\le n$ that
\begin{align*}
&h_{1,r}\tilde w_0 = \frac{[rm_1]_1}{r}c_1^r(\tilde w_0+rz_1w_0),\\
&h_{i,r}\tilde w_{i-1} = -\frac{[ra_{i,i-1}]_i}{r}c_i^r(\tilde w_{i-1}+rz_1w_{i-1}), \ \ 1<i<n.
\end{align*}
For $i=1$ this follows from the above, so induction starts. Suppose we have proved the above for $1\le i<n-1$. In particular, if $i>1$, it follows by applying Section \ref{remark} to $\hat{\bu}_q^iw_{i-1}$ that,
\begin{equation*}
x_{i,r}^-\tilde w_{i-1}= (c_iq_i^{-a_{i,i-1}})^r(\tilde w_i+rz_1w_i)= c_{i+1}^r(\tilde w_i+rz_1w_i).
\end{equation*}
The inductive step is completed by  using the preceding equation and noting that
$$h_{i+1,r}\tilde w_i= h_{i+1,r}x_{i,0}^- \tilde w_{i-1} = -\frac{[ra_{i+1,i}]_{i+1}}{r} x_{i,r}^- \tilde w_{i-1}.$$

Now observe that $\hat\bu_q^n$--submodule generated by $\tilde w_{n-1}+\hat{\bu}_q^nw_{n-1}$ is a self extension of $V(\pi(m_n-a_{n,n-1},c_nq_n^{a_{n,n-1}}))$. Hence, it follows from Section \ref{remark} that
\begin{equation}\label{eq:finalstep}
h_{n,r}\tilde w_{n-1} = (c_nq_n^{a_{n,n-1}})^r\frac{[r(m_n-a_{n,n-1})]_n}{r}(\tilde w_{n-1}+ z' rw_{n-1}) \quad\text{for some}\quad z'\in\bc.
\end{equation}
Further, since $[h_{n,r},x_{i,s}^-]=0$ if $i\ne n,n-1$, we see that
\begin{equation}
h_{n,r}\tilde w_i = c_n^r\frac{[rm_n]_n}{r}(\tilde w_i+ z_n rw_i) \quad\text{for all}\quad i<n-1.
\end{equation}
This implies
\begin{align*}
h_{n,r}\tilde w_{n-1} &=(-\frac{[ra_{n,n-1}]_{n}}{r} x_{n-1,r}^- + x_{n-1,0}^-h_{n,r})\tilde w_{n-2}\\
&=(-\frac{[ra_{n,n-1}]_{n}}{r}(c_{n-1}q_{n-1}^{-a_{n-1,n-2}})^r)(\tilde w_{n-1}+rz_1w_{n-1}) + \frac{[rm_n]_n}{r}c_n^r(\tilde w_{n-1}+ rz_nw_{n-1})\\
&= (c_nq_n^{a_{n,n-1}})^r\frac{[r(m_n-a_{n,n-1})]_n}{r}\tilde w_{n-1}\\
 &\ \ + r(-z_1\frac{[ra_{n,n-1}]_{n}}{r}(c_nq_{n}^{a_{n,n-1}}q_n^{-m_n})^r + z_n\frac{[rm_n]_n}{r}c_n^r) w_{n-1}.
\end{align*}
Comparing this with \eqref{eq:finalstep} we get
\begin{equation}
-z_1[ra_{n,n-1}]_{n}(q_{n}^{r(a_{n,n-1}-m_n)})^r + z_n[rm_n]_n = z'q_n^{ra_{n,n-1}}[r(m_n-a_{n,n-1})]_n
\end{equation}
for all $r\in\bz$. It follows that $z_1=z_n=z'$.
\hfill\qedsymbol
\bibliographystyle{amsplain}

\begin{thebibliography}{10}

\bibitem{AK} T. Akasaka and M. Kashiwara, {\em Finite--dimensional Representations of quantum qffine algebras}, Publ.Res.Inst. Math.Sci., {\bf 33}, 1997, no. 5, 839--867.

\bibitem{Beck}
J. Beck, {\em Braid group action and quantum affine algebras}, Commun. Math. Phys. {\bf 165} (1994), 555--568.

\bibitem{BCP}
J. Beck, V. Chari, and A. Pressley, {\em An algebraic characterization of the affine canonical basis}, Duke Math. J. {\bf 99} (1999), no. 3, 455--487.

\bibitem{BN}
J.~Beck and H.~Nakajima, {\em Crystal bases and two-sided cells of quantum affine algebras}, Duke Math. J. {\bf 123} no. 2 (2004), 335--402.

\bibitem{Bourbaki}
N. Bourbaki, Elements of Mathematics - Lie Groups and Lie Algebras: Chapters 1-3, 4-6, and 7-9, Springer (1998).

\bibitem{Cmin}
V. Chari, {\em Minimal affinizations of representations of quantum groups: the rank-2 case}, Publ. Res. Inst. Math. Sci. {\bf 31} (1995), 873--911.

\bibitem{Cbraid}
V.~Chari, {\em Braid group actions and tensor products}, Int. Math. Res. Notices (2002), no. 7, 357--382.

\bibitem{CFK}
V.~Chari, G.~Fourier, and T.~Khandai, {\em A categorical approach to Weyl modules}, Transformation Groups {\bf 15} (2010), 517--549.

\bibitem{CG}
V. Chari and J. Greenstein, {\em An application of free Lie algebras to polynomial current algebras and their representation theory},  Contemp. Math. {\bf 392} (2005), 15--31.


\bibitem{Cher} V. Chari and D. Hernandez, {\em Beyond Kirillov--Reshetikhin modules}, Comtemporary Mathematics, {\bf 506}, (2010), 49--81.
\bibitem{CL}
V. Chari and S. Loktev, {\em Weyl, Demazure and fusion modules for the current algebra of $\lie{sl}_{r+1}$}, Adv. Math., {\bf 207} (2006), no. 2, 928--960.


\bibitem{CPqa}
V.~Chari and A.~Pressley, {\em Quantum affine algebras}, Commun. Math. Phys. {\bf 142} (1991), 261--283.



\bibitem{CPbanff}
V.~Chari and A.~Pressley, {\em Quantum affine algebras and their representations}, Representations of Groups (Banff, AB, 1994), CMS Conf. Proc. {\bf 16} (1995), 59--78.

\bibitem{CPmin1}
V. Chari and A. Pressley, {\em Minimal affinizations of representations of quantum groups: the nonsimply laced case}, Lett. Math. Phys. {\bf 35} (1995), 99--114.

\bibitem{CPmin2}
V. Chari and A. Pressley, {\em Minimal affinizations of representations of quantum groups: the simply laced case}, J. Algebra {\bf 184} (1996), no. 1, 1--30.

\bibitem{CPmin3}
V. Chari and A. Pressley, {\em Minimal affinizations of representations of quantum groups: the irregular case}, Lett. Math. Phys. {\bf 36} (1996), 247--266.

\bibitem{CPWeyl}
V. Chari and A. Pressley, {\em Weyl modules for classical and quantum affine algebras}, Representation Theory {\bf 5} (2001), 191--223.

\bibitem{CPqweyl} V. Chari and A. Pressley, {\em Integrable and Weyl modules for quantum affine $sl_2$}, Quantum Groups and Lie theory, Proceedings of the LMS symposium on Quantum Groups, (1999) Durham, England.



\bibitem{Da}
I. Damiani, {\em La R-matrices pour les algèbres quantiques de type affine non tordu} [The R-matrix of nontwisted quantumm affine algebras], Ann. Sci. École Norm. Sup. {\bf 31} (1998), 493--523.

\bibitem{Dr}
V. Drinfeld, {\em A new realization of Yangians and quantum affine algebras}, Soviet. Math. Dokl. {\bf 36} (1988), 212--216.

\bibitem{FoL}
G. Fourier and P. Littelmann, {\em Weyl modules, Demazure modules, KR-modules, crystals, fusion products and limit constructions}, Adv. Math. {\bf 211} (2007), no. 2, 566--593.

\bibitem{FM}
E. Frenkel and E. Mukhin, {\em Combinatorics of $q$-characters of finite-dimensional representations of quantum affine algebras}, Comm. Math. Phys. {\bf 216} (2001), 23--57.


\bibitem{FR}
E.~Frenkel and N.~Reshetikhin, {\em The $q$-characters of representations of quantum affine algebras and deformations of $\mathcal{W}$-algebras},  Contemp. Math. {\bf 248} (1999),  163--205.

\bibitem{Her06}
D.~Hernandez, \emph{{The Kirillov-Reshetikhin conjecture and solutions of
  T-systems}}, J. Reine Angew. Math. \textbf{2006} (2006), 63--87.



\bibitem{her:simpletp}
D. Hernandez, {\em Simple tensor products}, Invent. Math. {\bf 181} (2010), 649--675.

\bibitem{herlec}
D. Hernandez and B. Leclerc, {\em Cluster algebras and quantum affine algebras}, Duke Math. J. {\bf 154} (2010), no. 2, 265--341.

\bibitem{K}
M. Kashiwara, {\em Crystal bases of modified quantized enveloping algebra}, Duke Math.J. {\bf 73} (1994), 383--413.

\bibitem{Ka}
M. Kashiwara, {\em On level zero representations of quantized affine algebras}, Duke Math. J. {\bf 112} no.1 (2002), 117--195.

\bibitem{Kodera}
R. Kodera, {\em Extensions between finite-dimensional simple modules over a generalized current Lie algebra},  Transform. Groups  {\bf 15}  (2010), 371--388.

\bibitem{KNS}
A.~Kuniba, T.~Nakanishi, J.~Suzuki,
\emph{Functional relations in solvable lattice models. I. Functional relations and representation theory}
Int. J. Mod. Phys. A\textbf{9}, (1994) no. 30, 5215--5266.

\bibitem{KNS2} A.~Kuniba, T.~Nakanishi, J.~Suzuki,
\emph{T-systems and Y-systems in integrable systems}
J. Phys. A: Math. Theor. \textbf{44} (2011), 103001.

\bibitem{Lusztig} G.~Lusztig, {Quantum deformations of certain simple modules over enveloping algebras}, Adv. in Math. {\bf 70} (1988), 237--249.

\bibitem{Md4}
A. Moura, {\em Restricted limits of minimal affinizations}, Pacific J. Math {\bf 244} (2010), no. 2, 359--397.

\bibitem{MY}
E. Mukhin, C. A. S. Young, {\em Extended T--systems}, to appear in Selecta Math., arXiv:1104.3094.

\bibitem{N} H. Nakajima, {\em Quiver varieties and  finite--dimensional representations of quantum affine algebras} J. Amer. Math. Soc.  {\bf 14} (2001), no.1, 145--238.
\bibitem{Na}
H. Nakajima, {\em Extremal weight modules of quantum affine algebras}, Advanced Studies in Pure Math. {\bf 40} (2004), 343--369

\bibitem{Nak03}
H.~Nakajima, \emph{{$t$-analogs of $q$-characters of Kirillov-Reshetikhin modules
  of quantum affine algebras}}, Represent. Theory \textbf{7} (2003), 259--274.

\bibitem{Nak:cluster}
H.~Nakajima, {\em Quiver varieties and cluster algebras}, Kyoto J. Math. {\bf 51} (2011), 71--126.


\bibitem{Naoi}
K. Naoi, {\em Weyl modules, Demazure modules and finite crystals for non-simply laced type}, to appear in Advances in Mathematics, arXiv:1012.5480.

\bibitem{VV}
M. Varagnolo and E. Vasserot, {\em Standard modules of quantum affine algebras},  Duke Math. J. {\bf 111} (2002), 509--533.

\end{thebibliography}

\end{document}